\documentclass[11pt,reqno]{amsart}
\usepackage[utf8x]{inputenc}
\usepackage[T1]{fontenc}
\usepackage{amsmath,esint}
\usepackage{amsthm}
\usepackage{amsfonts}\usepackage{amssymb}
\usepackage{a4wide}
\usepackage{color}

\newcommand{\ek}{{\rm e}_k}
\newcommand{\wt}{\widetilde}

\newcommand{\bea}{\begin{eqnarray}}
\newcommand{\eea}{\end{eqnarray}}
\newcommand{\bee}{\begin{eqnarray*}}
\newcommand{\eee}{\end{eqnarray*}}
\newcommand\NN{\mathbb{N}}
\newcommand\FF{\mathbf{F}}
\newcommand\RR{\mathbb{R}}

\newcommand\PP{\mathbf{P}}

\newcommand\EE{\mathcal{E}}
\newcommand\HH{\mathbf{H}}
\newcommand\JJ{\mathbf{J}}

\newcommand\yzk{y_{z_k}}
\newcommand\e{\varepsilon}
\newcommand\be{\begin{equation}}
\newcommand\ee{\end{equation}}
\newcommand\lsl{\frac{\dot \lambda}{\lambda}}
\newcommand\md{{\rm M}}
\newcommand\la{\langle}
\newcommand\ra{\rangle}
\newcommand\re{\mathrm{\, Re\,}}
\newcommand\im{\mathrm{\, Im\,}}
\newcommand\pa{\partial}
\newcommand\R{\RR}
\newcommand\lam{\lambda}
 
\newtheorem{theorem}{Theorem}
\newtheorem{lemma}[theorem]{Lemma}
\newtheorem{corollary}[theorem]{Corollary}
\newtheorem{proposition}[theorem]{Proposition}

\numberwithin{equation}{section}
 
\title[Blow up for critical NLS]{Strongly interacting blow up bubbles \\ for the mass critical NLS}
\author[Y. Martel]{Yvan Martel}
\address{CMLS, \'Ecole Polytechnique, CNRS, 91128 Palaiseau, France}
\email{yvan.martel@polytechnique.edu}
\author[P. Rapha\"el]{Pierre Rapha\"el}
\address{Universit\'e de Nice Sophia-Antipolis, Institut Universitaire de France, European Research Council, France}
\email{pierre.raphael@unice.fr}
\begin{document} 
\begin{abstract} We consider the mass critical two dimensional nonlinear Schr\"odinger equation  (NLS)
$$i\pa_tu+\Delta u+|u|^2u=0, \ \ t\in \RR,\  x\in \RR^2.$$ 
Let $Q$ denote the positive ground state solitary wave satisfying $\Delta Q-Q+Q^3=0$. We construct a new class of multi--solitary wave solutions:  given any integer $K\geq 2$, there exists a global (for $t>0$) solution $u(t)$ of (NLS) 
 that decomposes asymptotically into a sum of solitary waves centered at the vertices of a $K$--sided regular polygon and concentrating at a logarithmic rate as $t\to +\infty$ so that  the solution blows up in infinite time with the rate
 $$\|\nabla u(t)\|_{L^2}\sim |\log t|\ \ \mbox{as}\ \ t\to+\infty.$$ 
This special behavior is due to  strong interactions between the waves, in contrast with previous works on multi--solitary waves of (NLS) where interactions do not affect the blow up rate. Using the pseudo--conformal symmetry of the (NLS) flow, this yields the first example  of solution $v(t)$ of (NLS) blowing up in finite time  with a rate strictly above the pseudo--conformal one, namely, $$\|\nabla v(t)\|_{L^2}\sim \left| \frac{\log |t|}{t}\right|\ \ \mbox{as}\ \ t\uparrow 0.$$
Such solution concentrates $K$ bubbles at a point $x_0\in \RR^2$, i.e.
$|v(t)|^2\rightharpoonup K\|Q\|_{L^2}^2\delta_{x_0}$ as $t\uparrow 0.$
\end{abstract}
\maketitle


\section{Introduction}


\subsection{General setting}

We consider in this paper the mass critical two dimensional non linear Schr\"odinger equation (NLS)
\be\label{nls}
i\pa_tu+\Delta u+|u|^{2}u=0,\quad (t,x)\in \RR\times\RR^2.
\ee
It is well-known (see e.g. \cite{Ca03} and the references therein) that for any  $u_0\in H^1(\RR^2)$, there exists a unique maximal solution $u\in\mathcal{C}((-T_\star,T^\star),H^1(\RR^2))$ of~\eqref{nls} with $u(0)=u_0$. Moreover, the following blow up criterion holds
\be
\label{blowoucrit}
 T^\star<+\infty\ \ \mbox{implies}\ \ \lim_{t\uparrow T^\star}\|\nabla u(t)\|_{L^2}=+\infty.
 \ee
The mass (i.e. the $L^2$ norm) and the energy $E$ of the solution are conserved by the flow, where
\begin{equation*}
E(u)=\frac12\int_{\RR^2} |\nabla u|^2-\frac{1}{4}\int_{\RR^2} |u|^4
\end{equation*}
From a variational argument, the unique (up to symmetry) ground state solution to 
\[
\Delta Q-Q+Q^3=0, \ \ Q\in H^1(\RR^2), \ \ Q>0, \ \ \mbox{$Q$ is radially symmetric}
\]
attains the best constant $C$ in the following Gagliardo--Nirenberg inequality 
\begin{equation}\label{GN}
\forall u\in H^1(\RR^2), \ \  \|u\|_{L^4}^{{4}}\leq C\|u\|_{L^2}^{2}\|\nabla u\|_{L^2}^2
\end{equation}
(see  \cite{BL83,We83,Kw89}).
As a consequence, one has
\be
\label{gagenergy}
\forall u\in H^1(\RR^2), \ \ E(u)\geq \frac 12\|\nabla u\|_{L^2}^2\left(1- \frac{\|u\|_{L^2}^2}{\|Q\|_{L^2}^2}\right).
\ee
Together with the conservation of mass and energy and the blow up criterion~\eqref{blowoucrit}, this implies the global existence of any solution with initial data $\|u_0\|_{2}<\|Q\|_{2}$. Actually it is now known that in this case, the solution scatters i.e. behaves asymptotically in large time as a solution of the linear equation,
 see \cite{KTV,Do} and references therein.\\
We also know that  $\|u\|_{L^2}=\|Q\|_{L^2}$ corresponds to  the mass threshold for global existence since the pseudo--conformal symmetry of the (NLS) equation
\be
\label{pseduocno}
v(t,x)=\frac{1}{|t|}u\left( \frac{1}{|t|},\frac{x}{|t|}\right)e^{-i\frac{|x|^2}{4|t|}}
\ee applied to the solitary wave solution $u(t,x)=e^{it} Q(x)$  yields the existence of an explicit single bubble blow up solution $S(t)$ with minimal mass 
\be\label{defst}
S(t,x) = \frac{1}{|t|}Q\left(\frac {x}{|t|} \right)e^{-i\frac{|x|^2}{4|t|}}e^{\frac {i}{|t|}},\quad
\|S(t)\|_{L^2}= \|Q\|_{L^2}, \quad \|\nabla S(t)\|_{L^2}\mathop{\sim}_{t\sim 0^-}\frac 1{|t|}.
\ee
We refer to \cite{Ca03} for more properties of the pseudo-conformal transform.
From \cite{Me93},   minimal mass blow up solutions are {\it classified} in $H^1(\R^2)$: $$\|u(t)\|_{L^2}=\|Q\|_{L^2}\ \ \mbox{and}\ \ T^*<+\infty\ \ \mbox{imply}\ \ u\equiv S \ \mbox{up to the symmetries of the flow.}$$  Recall also the following well-known general sufficient criterion for finite time blow up: for initial data $u_0\in \Sigma=H^1\cap L^2(|x|^2dx)$, the virial identity 
\be
\label{globalvirial}
\frac{d^2}{dt^2}\int_{\RR^2} |x|^2|u(t,x)|^2dx=16E(u_0)
\ee
implies  blow up in finite time provided $E(u_0)<0$ (by~\eqref{gagenergy}, this implies necessarily $\|u_0\|_{L^2}>\|Q\|_{L^2}$).

\subsection{Single bubble blow up dynamics} 
We focus now on the case of mass slightly above the threshold, i.e.
\be
\label{condmass}
\|Q\|_{L^2}<\|u_0\|_{L^2}<\|Q\|_{L^2}+\alpha_0,  \ \ 0<\alpha_0\ll1 .
\ee 
We first recall in this context that a large class of finite time blow up solutions was constructed in \cite{BW} (see also \cite{KS}, \cite{MeRaSz14}) as weak perturbation of the minimal mass solution $S(t)$. In particular, these solutions blow up with the pseudo--conformal blow up rate  
\begin{equation}\label{conf} \|\nabla u(t)\|_{L^2}\mathop{\sim}_{t\sim T^*} {\frac{1}{T^*-t}}.
\end{equation}
Second, recall that the series of works \cite{Pe1,MeRa04,MeRa05,Ra05,MeRa05bis,MeRa06} 
provides a thorough study of the  {\it stable} blow up dynamics under condition~\eqref{condmass},
corresponding to the so called \emph{log--log} blow up regime 
\begin{equation}\label{llog} \|\nabla u(t)\|_{L^2}\mathop{\sim}_{t\sim T^*} c^* \sqrt{\frac{\log|\log(T^*-t)|}{T^*-t}}.
\end{equation}
Third, it is proved in \cite{MeRaSz14} (see also \cite{KS}) that   solutions constructed in \cite{BW} are unstable and correspond in some sense to a threshold between the above log--log blow up and scattering.

\noindent Finally, recall that under~\eqref{condmass}, a universal gap on the blow up speed  was proved in \cite{Ra05}: given a finite time blow up solution satisfying~\eqref{condmass},  either it blows up in the log--log regime~\eqref{llog}, or it blows up faster than the pseudo--conformal rate
$$\|\nabla u(t)\|_{L^2}\gtrsim \frac1{T^*-t}.$$ 
(See also \cite{An,Ba}.)
However,  the  existence of solutions blowing up strictly faster than the   conformal speed 
 is a long lasting open problem, which is equivalent, by   the pseudo--conformal symmetry~\eqref{pseduocno}, to the existence of global solutions blowing up in infinite time.

\subsection{Multi bubbles blow up dynamics} For  larger $L^2$ mass, it is conjectured (see e.g. \cite{MeRa07}) that any finite time blow up solution   concentrates at the blow up time universal quanta of mass $m_j>0$ at a finite number of points $x_k\in \RR^2$ , i.e.
 $$|u(t)|^2\rightharpoonup \sum_{k=1}^K m_k \delta_{x_k}+|u^*|^2 \ \ \hbox{as } t\uparrow T^*,$$ 
 where $u^*\in L^2$ is a (possibly zero) residual. 
 The first example of multiple point blow up solution  is given in \cite{Mmulti}: let $K\geq 1$ and let  $(x_k)_{1\leq k\leq K}$ be $K$ arbitrary distinct points of $\RR^2$,  
 there exist a finite time blow up solution $u(t)$ of~\eqref{nls} with $$
 \left\|u(t)-\sum_{k=1}^K S(t,.-x_k)\right\|_{H^1}\to 0,\quad 
 |u(t)|^2\rightharpoonup \|Q\|_{L^2}^2 \sum_{k=1}^K \delta_{x_k}\ \ \mbox{as}\ \ t\uparrow 0.$$  
 In particular, $u(t)$ blows up with the pseudo--conformal rate
 $$\|\nabla u(t)\|_{L^2}\sim \frac{1}{|t|}\ \ \hbox{as } t\uparrow 0. $$ 
Other general constructions of multi bubble blow up are provided by \cite{PlRa07,Fan}
in the context of the log--log regime. 
Observe that these works deal with {\it weak} interactions  in the sense that the blow up dynamics of each bubble is not perturbed at the main order by the presence of the other (distant) bubbles.

\subsection{Main results} 
In this paper we construct the first example of infinite time blow up solution of (NLS),
related to the  \emph{strong}   interactions of an arbitrary number $K\geq 2$ of bubbles.
As a consequence,  
using the pseudo--conformal transform, we also obtain the first example of solution blowing up in finite time strictly faster than the conformal blow up rate.
 Such a solution concentrates the $K$ bubbles at one point at the blow up time.

\begin{theorem}[Infinite time blow up]
\label{thmmain} Let $K\ge 2$ be an integer. 
There exists a solution $u\in \mathcal C([0,+\infty),\Sigma)$ of~\eqref{nls} 
which decomposes asymptotically into a  sum of $K$ solitary waves 
\be
\label{strognconv}
\left\|u(t)- e^{i\gamma(t)}\sum_{k=1}^K\frac{1}{\lam(t)}Q\left(\frac{.-x_k(t)}{\lam(t)}\right)\right\|_{H^1}\to 0,
\quad \lam(t) = \frac {1+o(1)}{\log t}  \ \ \mbox{as} \ \ t\to +\infty,
\ee 
where the translation parameters $x_k(t)$  
  converge as $t\to +\infty$ to the vertices of a $K$-sided regular polygon, and where $\gamma(t)$ is  some phase parameter.
 In particular,
\be\label{thm2}
 \|\nabla u(t)\|_{L^2} = K^{\frac 12}\|\nabla Q\|_{L^2}  (1+o(1)) \log t \ \
 \mbox{as} \ \ t\to +\infty.
\ee
\end{theorem} 

\begin{corollary}[Finite time collision]
\label{maincor}
Let $u(t)\in C([0,+\infty),\Sigma)$ be given by Theorem~\ref{thmmain} and 
let $v\in \mathcal C((-\infty,0),\Sigma)$ be the pseudo conformal transform of $u(t)$ defined by~\eqref{pseduocno}.
Then $v(t)$  blows up at $T^*=0$   with
\be
\label{logmodifiy}
\|\nabla v(t)\|_{L^2}= K^{\frac 12}\|\nabla Q\|_{L^2}(1+o(1))\left| \frac{\log |t|}{t}\right|,
\quad |v|^2\rightharpoonup K\|Q\|_{L^2}^2\delta_{0}\ \ \mbox{as}\ \ t\uparrow 0.
\ee
\end{corollary}

\noindent{\it Comments on the main results.}

\medskip

\noindent{\it 1. Dynamics with multiple nonlinear objects.} 
Multiple bubble solutions with \emph{weak interactions} and  asymptotically free Galilean motion  have been constructed in various settings, both in stable and unstable contexts, see in particular \cite{Mmulti,Mi03,Martelkdv,MMnls,KMR,CMM,Perelman,Boulenger,Fan}. 
As a typical example of weakly interacting dynamics, for the nonlinear Schr\"odinger equations 
\begin{equation}\label{nlsp}
i\partial_t u + \Delta u + |u|^{p-1} u = 0,\quad x\in \RR^d, \quad 1<p<1+\frac {4}{d-2},
\end{equation}
there exist   multi solitary wave solutions satisfying for large $t$,
\be
\label{multisolitary}
\left\|u(t)-\sum_{k=1}^K e^{-i \Gamma_k(t,x)} \omega_k^{\frac 1{p-1}}Q\Big(\omega_k^{\frac 12}(.-\nu_k t)\Big)
\right\|_{H^1}\lesssim e^{-\gamma t}, \ \ \gamma>0,
\ee
for any given set of parameters $\{\nu_k,\omega_k\}_k$ with the decoupling condition $\nu_k\neq \nu_{k'}$ if $k\neq k'$
(see \cite{MMnls,CMM}).

\noindent 
In \cite{KMR}, two different regimes with \emph{strong interactions} related to the two body problem of gravitation
are exhibited for the Hartree model (hyperbolic  and parabolic asymptotic motions). We  also refer to \cite{Mi03,MaMeinvent} for works related to sharp interaction problems in the setting of the  subcritical (gKdV) equation. 
We thus see the present work as the first intrusion into the study of strongly interacting non radial multi solitary wave motions for (NLS). Note that the solution given by Theorem~\ref{thmmain} is a minimal threshold dynamics and its behavior is unstable by perturbation of the data. An important direction of further investigation is the derivation of  {\it stable} strongly interacting multiple bubbles blow up dynamics.

\noindent We observe from the proof of Corollary~\ref{maincor} that the $K$ bubbles of the solution  collide at the same point at the blow up time  providing  the first example of  collision at blow up for (NLS). Note that the geometry of the  trajectories of the blow up points (straight lines from the origin to the egde of the $K$-sided regular polygon) is an essential feature of these solutions. 
A related one dimensional mechanism is involved in the derivation of degenerate blow up curves in the context of ``type I'' blow up for the wave equation, see \cite{MZ}.
For the nonlinear heat equation in one dimension, solutions for which two points of maximum collide at blow up are constructed in \cite{HV}.
There are also analogies of the present work with the construction of stationary solutions with mass concentrated along specific nonlinear grids, see \cite{MPW}.
In the context of two dimensional incompressible fluid mechanics, special solutions to the vortex point system are studied as a simplified model for dynamics of interacting and possibly colliding vortex, see for example \cite{Miot} for an overview of these problems.
 
\medskip

\noindent{\it 2. Minimal mass solutions.}
The proof of Theorem~\ref{thmmain} follows the now standard strategy of constructing {\it minimal}  dynamics by approximate solutions and compactness, initiated in \cite{Mmulti} and extended in various ways and  contexts by \cite{Martelkdv,KMR,CMM,RaSz11}.
We combine in a blow up context the approach developed for multibubble flows  in \cite{Martelkdv,KMR} and  a specific strategy to construct minimal blow up solutions for (NLS) type equations 
introduced  in \cite{RaSz11,KLR}.
A key ingredient of the proof is the precise tuning of the interactions between the waves.
In particular, we observe that the $K$ bubbles in \eqref{strognconv}  have the same phase, which is crucial in our analysis. 
The dynamics of  two symmetric bubbles with opposite phase ($\gamma_1=\gamma_2+\pi$) is related to  the dynamics of a single bubble  on a half--plane with Dirichlet boundary condition and it is known in this context  that minimal mass blow up at a boundary point (which  corresponds to the collision case) does not exists, see \cite{Ba}.

\noindent Note that we restrict ourselves to space dimension $2$ for simplicity, but similar results hold for the mass critical (NLS) equation in any space dimension with   same proof.

\noindent
For the mass subcritical ($1<p<1+\frac 4d$) and supercritical ($1+\frac 4d<p<1+\frac 4{d-2}$) 
nonlinear Schr\"odinger equations \eqref{nlsp},
we expect using similar approach the existence of  bounded  strongly interacting multi solitary waves,
with logarithmic relative distances, i.e.  non free Galilean motion. 
Interestingly enough, the existence of such solutions  is {\it ruled out} in the mass critical case  by the   virial law~\eqref{globalvirial}. The scaling instability direction of the critical case is excited by the interactions which leads to the infinite time concentration displayed in Theorem~\ref{thmmain}.

\noindent Conversely, solutions such as $u(t)$  in Theorem~\ref{thmmain} cannot exist in the sub and supercritical cases. In the subcritical case, it is well-known that all $H^1$ solutions are bounded in $H^1$ (\cite{Ca03}). In the supercritical case, any solution in $\Sigma$ that is global for $t\geq 0$ satisfies
\[
\liminf_{t\to +\infty} \|\nabla u(t)\|_{L^2} \lesssim 1.
\]
Indeed in this case,  the Virial identity $\frac {d^2}{dt^2} \int |x|^2 |u|^2 = 4d(p-1) E(u_0) - (\frac d2(p-1)-1)\int |\nabla u|^2$
integrated twice in time provides the global bound $\int_0^t \int_0^s \|\nabla u(s)\|_{L^2}^2 ds dt \lesssim t^2.$

\noindent Note that the construction of Theorem~\ref{thmmain} is performed near $t=+\infty$ (by translation invariance, it is then obvious to obtain a solution  on the time interval $[0,+\infty)$). An interesting  question is to understand the behaviour of such solutions for $t\leq 0$.

\medskip

 \noindent{\it 3. Zero energy global solutions}. From the proof of Theorem~\ref{thmmain}, the solution $u$ has zero energy.  
In \cite{MeRa06}, it is proved that any zero energy solution satisfying~\eqref{condmass} blows up in finite time with the log--log regime. Thus, in the neighborhood of  $Q$, $e^{it} Q$  is the only global zero energy solution.
For the critical (gKdV) equation, a similar result holds, though in a stronger topology (see \cite{MMR1}). 
Note that the existence of global in time zero energy solutions is strongly related to Liouville type  theorems and to blow up profile, see \cite{Nawa,MeRa04}. 
For (NLS), the only known examples of global in positive time zero energy solutions so far were the time periodic solutions $ e^{it} P$ where  $P$ is any solution to the stationary equation $\Delta P-P+P^3=0$.
Therefore, the existence of such a non trivial global (for positive time) zero energy solution $u(t)$ is surprising.
For other works related to minimal mass solutions and their key role in the dynamics of the flow, we   refer to 
 \cite{Mdu,RaSz11,BCD,MeRaSz14,DKM,MMR1,MMRII}.
 
\medskip

\noindent{\it 4. Blow up speed for \emph{(NLS)}}. The question of determining all possible blow up rates for solutions of nonlinear dispersive equations is in general intricate. 
For the (NLS) equation~\eqref{nlsp} in the  mass supercritical--energy subcritical range, a universal sharp upper bound on the blow up rate has been derived in \cite{MeRaSz14} for radial data, but no such bound exists for the mass critical problem.  
For (NLS) with a double power non linearity of the form $|u|^{p-1}u+|u|^2u$ where $1<p<3$, the minimal mass solution has a surprising blow up rate different from the conformal rate, see \cite{LeCMR}.
For the mass critical (gKdV) equation, solutions arbitrarily close to the solitary wave with arbitrarily fast blow up speed have been constructed in \cite{MMRIII}.
Recall that constructions of   blow up solutions with various blow up rate are  also available in the energy critical and super--critical context, see  \cite{KST,KS2,MRR,DHKS,Collot,JJ}. 
However, such general constructions seem by now out of reach  for the mass critical (NLS) problem. 
In this context, the derivation of the anomalous blow up speed~\eqref{logmodifiy}, in spite of its  rigidity, is 
an interesting new fact. We will see in the proof how such a blow up rate is  related to strong coupling between the solitary waves.

\subsection{Notation}
Let $\Sigma=H^1\cap L^2(|x|^2dx)$.
The $L^2$ scalar product of two complex valued functions $f,g\in L^2(\RR^2)$ is denoted by
\[
\la f,g\ra = \re\left( \int_{\RR^2} f(x) \overline g(x) dx\right).
\]
In this paper, $K$ is an integer with $K\geq 2$. For brevity, $\sum_k$ denotes $\sum_{k=1}^K$. 
For $k=1,\ldots, K$, $\ek$ denotes the unit vector of $\RR^2$ corresponding to the complex number $e^{i\frac {2\pi(k-1)}{K}}$. We define the constant $\kappa=\kappa(K)$ by
\be\label{d:ka}
\kappa = \left|1 - e^{i \frac{2\pi}{K}}\right| = (2-2\cos(2\pi/K))^{1/2} >0.
\ee
Recall that we denote by $Q(x):=Q(|x|)$ the unique radial positive ground state of~\eqref{nls}:

\be\label{e:q}
Q''+\frac{Q'}{r}- Q+Q^3 = 0,\quad Q'(0)=0,\quad \lim_{r\to +\infty}Q(r)=0.
\ee
It is well-known and easily checked by ODE arguments that for some constant $c_Q>0$,
\be\label{as:q}
\hbox{for all $r>1$},\quad 
\left| Q(r)-c_Q r^{-\frac 12} e^{-r}\right| +
\left| Q'(r)+c_Q r^{-\frac 12} e^{-r}\right| \lesssim r^{-\frac 32}e^{-r}.
\ee
We set
\be\label{d:Iq}
I_Q = \int   Q^3(x) e^{x_1} dx, \quad x=(x_1,x_2).
\ee
We denote by $\mathcal Y$ the set of smooth functions $f$ such that
\be\label{d:Y}
\hbox{for all $p\in \NN$, there exists $q\in \NN$, s.t. for all $x\in \RR^2$}\quad |f^{(p)}(x)|\lesssim |x|^q e^{-|x|}.
\ee
Let  $\Lambda$ be the generator of $L^2$-scaling in two dimensions:
\[
\Lambda f=f+x\cdot\nabla f  .
\]
The linearization of~\eqref{nls} around  $Q$ involves the following Schr\"odinger operators:
\[
L_+:=-\Delta+1-3Q^2,\qquad
L_-:=-\Delta+1-Q^2.
\]
Denote by $\rho\in \mathcal Y$   the unique radial solution $H^1$ to 
\be
\label{eqrho}
L_+\rho =\frac{|x|^2}4Q
\ee which satisfies on $\RR^2$
\be\label{d:rh}
|\rho(x)|+|\nabla\rho(x)|\lesssim (1+|x|^3) Q(x).
\ee
We recall the generalized null space relations  (see \cite{We85})
\begin{equation}\label{e:s}
\begin{aligned}
&L_-Q=0,\quad L_+(\Lambda Q)=-2Q,\quad L_-(|x|^2Q)=-4\Lambda Q,\quad L_+\rho =\frac{|x|^2}4Q,\\
&L_+(\nabla Q)=0,\quad L_-(xQ)=- 2 \nabla Q,
\end{aligned}
\end{equation}
and the classical (see e.g. \cite{MeRa04,MeRa05,RaSz11,We85,ChGuNaTs07}) coercivity property: there exists $\mu>0$ such that  for all $\eta\in H^1$, 
\begin{multline}\label{e:co} 
\la L_+ \re \eta,\re \eta\ra + \la L_-\im\eta,\im\eta\ra \geq \mu \|\eta\|_{H^1}^2\\
-\frac 1\mu\left( \la \eta,Q \ra^2+\la \eta,|x|^2Q \ra^2+|\la \eta,xQ \ra|^2
+\la \eta,i \rho\ra^2+|\la\eta,i\nabla Q \ra|^2 \right).
\end{multline}


\subsection{Outline of the paper}

The main goal of Sect.~\ref{s:2} is to
 construct a symmetric  $K$--bubble approximate solution to (NLS) and to extract the formal evolution system
of the geometrical parameters of the bubbles.
The key observation is that this system contains forcing terms due to the nonlinear interactions of the waves, and has a special solution corresponding at the main order to the regime of Theorem~\ref{thmmain}
(see Sect.~\ref{s:2.2}).
In Sect.~\ref{s:3}, we prove  uniform estimates 
on particular \emph{backwards solutions} of (NLS) related to the special   regime of Theorem~\ref{thmmain}.
We proceed in two main steps. First, we control the residue term by  energy arguments in the context of multi--bubbles. Second,   a careful adjustment of the \emph{final data} yields a  uniform control of the geometrical parameters. 
In Sect.~\ref{s:4}, we finish the proof of Theorem~\ref{thmmain} by   compactness  arguments on a suitable sequence of {backwards solutions} of (NLS) satisfying the uniform estimates of Sect.~\ref{s:3}.


\subsection*{Acknowledgements}  Y.M. was partly supported by ERC 291214 BLOWDISOL.
P.R. was supported by the ERC-2014-CoG 646650 SingWave and the Institut Universitaire de France. 
This work was done while the authors were members of the MSRI (Fall 2015).
They warmly thank MSRI for its hospitality.
The authors thank Thierry Cazenave for enlightening discussions.

\section{Approximate solution}\label{s:2}

In this section, we first
 construct a symmetric  $K$--bubble approximate solution to (NLS) and extract the evolution system
of the geometrical parameters of the bubbles. This system contains forcing terms due to the nonlinear interactions of the waves.
Second, we  write explicitly a special formal solution of this system that will serve as a guideline for the construction of the special solution $u(t)$ of Theorem~\ref{thmmain}.
Third, we state a standard modulation lemma around the approximate solution.
Recall that the integer $K\geq 2$ is fixed.

\subsection{Approximate solution and nonlinear forcing}
Consider a time dependent $C^1$ function $\vec p$ of the form
\[
\vec p = (\lambda,z,\gamma,\beta,b)\in (0,+\infty)^2\times \RR^3,
\]
with $|b|+|\beta|\ll 1$ and $z\gg 1$. We renormalize the flow by considering
\be\label{d:v}
u(t,x)=\frac {e^{i\gamma(s)}}{\lambda(s)} v(s,y),\quad   {dt} =  {\lambda^2(s)} {ds},
\quad y=\frac x{\lambda(s)},
\ee
so that
\be\label{e:v}
i \partial_t u + \Delta u + |u|^2 u=
\frac{e^{i\gamma}}{\lambda^3} \left[i\dot v + \Delta v - v + |v|^2 v - i\frac{\dot \lambda}{\lambda} \Lambda v + (1-\dot \gamma)v \right]
\ee
($\dot v$ denotes derivation with respect to $s$).
We  introduce the following $\vec p$-modulated ground state solitary waves, for any $k\in \{1,\ldots,K\}$, 
\begin{equation}\label{d:Pk}
 P_k(s,y) =  e^{i\Gamma_k(s,y-z_k(s))} Q_{a(z(s))}(y-z_k(s)), 
\end{equation}
for
\begin{equation}\label{d:bzk}
\beta_k =\beta \ek, \quad z_k = z \ek,\quad  \Gamma_k(s,y) =  \beta_k \cdot y - \frac b4 |y|^2,
\end{equation}
and where we have fixed
\be\label{d:a} 
Q_a = Q + a \rho,\quad 
a(z) = -c_a z^{\frac 12} e^{-\kappa z}, \quad
c_a =\left\{ 
\begin{aligned} 
&\frac{\kappa^{\frac 12}  c_Q  I_Q }{4 \la \rho,Q\ra} >0 \quad  \hbox{for $K=2$}\\
&\frac{\kappa^{\frac 12}  c_Q  I_Q }{2 \la \rho,Q\ra} >0 \quad  \hbox{for $K\geq 3$}
\end{aligned}\right.\ee
Note that the introduction of such modulated $Q_a$ corresponds to the intrinsic instability of the pseudo--conformal blow up regime ($a=0$ leads to $b(s)=s^{-1}$). Similar exact $Q_a$ (at any order of $a$) were introduced in \cite{MeRaSz13}.
The explicit above choice of $a(z)$ corresponds to direct integration of the nonlinear interactions at the main order, as explained in  Sect.~\ref{s:2.2}. We also refer to~\eqref{c:a} in the proof of Lemma~\ref{l:m} where this choice of $a(z)$ leads to an almost conservation of the mass for the approximate solution $\PP$ defined below. Note   that the different formula for $c_a$ depending on the value of $K$ corresponds to the fact that for $K\geq 3$, each given soliton has exactly two closest neighbor solitons.

\noindent Let
\be\label{d:PP}
\PP(s,y) =  \PP(y;(z(s),b(s),\beta(s)))= \sum_k P_k(s,y).
\ee
Then,  $\PP$ is an approximate solution of the rescaled equation  in the following sense.
\begin{lemma}[Leading order approximate flow]
\label{pr:P}
Let the vectors of modulation equations be 
\be\label{d:mkM}
\vec m_k^a=\left(
\begin{array}{c}
b+\lsl \\[0.1cm]
\dot z_k - 2\beta_k +\lsl z_k\\[0.1cm]
\dot \gamma - 1+ |\beta_k|^2 - \lsl (\beta_k \cdot z_k) - (\beta_k \cdot\dot z_k)  \\[0.1cm]
\dot \beta_k - \lsl \beta_k +  \frac b2 (\dot z_k - 2 \beta_k + \lsl z_k) \\[0.1cm]
\dot b + b^2 - 2 b (b+\lsl) -a 
\end{array}\right) ,
\qquad
\vec\md V =\left(
\begin{array}{c}
-i\Lambda V \\[0.15cm]
-i \nabla V\\[0.15cm]
-V\\[0.15cm]
-y V \\[0.15cm]
\frac {|y|^2}4 V
\end{array}
\right).
\ee
Then the error to the renormalized flow~\eqref{e:v} at $\PP$,
\be\label{d:EE}
 \EE_{\PP} = i\dot \PP + \Delta \PP - \PP + |\PP|^2 \PP - i \lsl \Lambda \PP + (1-\dot \gamma) \PP 
\ee
decomposes as 
\begin{equation}\label{d:psik}
\EE_{\PP} = \sum_k \big[e^{i\Gamma_k}\Psi_k\big] (y-z_k),
\quad
\Psi_k = \vec m_k^a \cdot \vec \md Q_a +i \dot  z a'(z) \rho + G_k +\Psi_{Q_a},
\end{equation}
where
\begin{equation}\label{b:Gk}
\|G_{k} \|_{L^\infty} \lesssim z^{-\frac 12} e^{-\kappa z},\quad 
\|\Psi_{Q_a}\|_{L^\infty}\lesssim |a|^2,
\end{equation}
and
\begin{equation}\label{b:iGk}
\left|\la G_k,iQ_a\ra + \kappa c_a \la \rho,Q\ra b z^{\frac 32} e^{-\kappa z}\right| \lesssim
(|\beta|^2 z^2+|b|^2z^4 + |\beta| z+|b|z) z^{-\frac 12} e^{-\kappa z} 
+ z^3 e^{-2 \kappa z}.
\end{equation}
\end{lemma}
\begin{proof}[Proof of Lemma~\ref{pr:P}]
\noindent \textbf{step 1.} Equation for $P_k$.
Let
\[
\mathcal E_{P_k} = i\dot P_k + \Delta P_k - P_k + |P_k|^2 P_k - i \lsl \Lambda P_k + (1-\dot \gamma) P_k.
\]
Let $\yzk=y-z_k$.
By direct computations
\begin{multline*}
i \dot P_k  = \Big[e^{i\Gamma_k} \Big( i\dot z a'(z) \rho - (\dot \beta_k\cdot \yzk) Q_a +(\dot z_k\cdot \beta_k) Q_a
\\ +\frac {\dot b} 4 |\yzk|^2 Q_a - \frac b 2 (\dot z_k \cdot \yzk) Q_a - i (\dot z_k \cdot \nabla Q_a)\Big)\Big] (\yzk),
\end{multline*}
\begin{multline*}
\Delta P_k  = \Big[e^{i\Gamma_k} \Big( \Delta Q_a - |\beta_k|^2 Q_a - \frac{b^2}4 |\yzk|^2 Q_a - i  b Q_a\\
+b (\beta_k\cdot \yzk) Q_a + 2 i (\beta_k\cdot \nabla Q_a) - i b (\yzk\cdot \nabla Q_a)\Big)\Big] (\yzk),
\end{multline*}
\begin{multline*}
\Lambda P_k  = \Big[e^{i\Gamma_k} \Big(  \Lambda Q_a + i(\beta_k\cdot \yzk) Q_a -i \frac b2 |\yzk|^2 Q_a
+(\yzk\cdot \nabla Q_a) \\+ i (z_{k}\cdot \beta_k) Q_a - i \frac b2 (z_k\cdot \yzk) Q_a + (z_k\cdot \nabla Q_a) \Big)\Big] (\yzk).
\end{multline*}
Thus,
\begin{align*}
& \mathcal E_{P_k}   = 
\Big[e^{i\Gamma_k} \Big( -i(b+\lsl) \Lambda Q_a - i (\dot z_k - 2 \beta_k+\lsl z_k) \cdot \nabla Q_a
\\& 
-(\dot \gamma - 1 +|\beta_k|^2 - \lsl (\beta_k\cdot z_k) - (\beta_k\cdot \dot z_k)) Q_a
 - (\dot \beta_k - \lsl \beta_k + \frac b2 (\dot z_k - 2 \beta_k + \lsl z_k)) \cdot \yzk Q_a\\
& +\frac 14 (\dot b + b^2 - 2 b (b+\lsl)) |\yzk|^2 Q_a
+i \dot  z a'(z) \rho 
+ \Delta Q_a - Q_a + |Q_a|^2 Q_a\Big)\Big]  (\yzk) 
\end{align*}
By $\Delta Q - Q + Q^3=0$ and  the definition  of $\rho$, $L_+ \rho = -\Delta \rho + \rho - 3 Q^2 \rho = \frac {1}{4} |y|^2 Q$ (see~\eqref{eqrho}), we have
\[
\Delta Q_a - Q_a + |Q_a|^2 Q_a
= - \frac a4 |y|^2 Q_a + \Psi_{Q_a},
\]
where
\begin{equation}\label{d:PQa}
\Psi_{Q_a} = |Q_a|^2 Q_a -Q^3 - 3 a Q^2 \rho + \frac {a^2}4 |y|^2 \rho.
\end{equation}
We have thus obtained the $P_k$ equation
\begin{equation}\label{e:Pk}
\mathcal E_{P_k}   = 
\Big[ e^{i\Gamma_k} \Big(\vec m_k^a \cdot \vec \md Q_a +i \dot  z a'(z) \rho  +\Psi_{Q_a} 
\Big)\Big]  (y-z_k),
\end{equation}
where $\vec m_k^a$ and $\vec \md$ are defined in~\eqref{d:mkM}.

\smallskip

\noindent\textbf{step 2.} Equation for $\PP$.  
From step 1 and the definition of $\EE_{\PP}$ in~\eqref{d:EE}, it follows that
\[
\EE_{\PP}  = \sum_k \EE_{P_k} + |\PP|^2 \PP - \sum_k |P_k|^2 P_k.
\]
Observe that
\[
|\PP|^2 \PP - \sum_k |P_k|^2 P_k 
= \sum_{j,\, k,\,  l} P_k P_j \overline P_l - \sum_k |P_k|^2 P_k  = \sum_k F_k,
\]
with
\[
F_k
= 2|P_k|^2  \sum_{j\neq k}  P_j + P_k^2 \sum_{j\neq k} \overline P_j  
+  \overline P_k \sum_{j\neq k,\, l\neq k,\, j\neq l} P_j P_l = \big[e^{i\Gamma_k} G_k\big](y-z_k),
\]
where we have set
\be\label{d:Gk}
G_k   = 2 G_{k}^{\rm (I)} + \overline G_{k}^{\rm (I)} + G_k^{\rm (II)},
\ee
and
\[\begin{split}
& G_{k}^{\rm (I)}(y)= \big[e^{-i \Gamma_k}Q_a^2\big](y) \sum_{j\neq k}  \big[ e^{i\Gamma_j} Q_a\big] (y-(z_j-z_k))	,\\
& G_k^{\rm (II)}(y) = \big[e^{-2i \Gamma_k} Q_a\big](y)     \sum_{j\neq k,\,  l\neq k,\,  j\neq l} 
\left( \big[ e^{i\Gamma_j}Q_a\big](y-(z_j-z_k))\cdot \big[ e^{i\Gamma_l}  Q_a\big](y-(z_l-z_k))\right).
\end{split}\]
Therefore,
\begin{equation}\label{d:psikb}
\EE_{\PP} = \sum_k \big[e^{i\Gamma_k}\Psi_k \big](y-z_k)
\quad \hbox{where}\quad
\Psi_k = \vec m_k^a \cdot \vec \md Q_a +i \dot  z a'(z) \rho + G_k +\Psi_{Q_a}.
\end{equation}

\noindent{\bf step 3} Nonlinear interaction estimates. In order to estimate the various  terms in~\eqref{d:psikb}, we will use the following interaction estimates: let $\omega,\wt\omega\in \RR^2$, $|\omega|\gg 1$, $|\wt\omega|\gg 1$, let $q\geq 0$, then:
\be\label{e:es}
\int (1+|y|^q) Q^3(y)Q(y-\omega)dy \lesssim |\omega|^{-\frac 12} e^{-|\omega|}.
\ee
\be\label{e:esb}
\int (1+|y|^q) Q^2(y)Q(y-\omega)Q(y-\wt\omega)dy \lesssim  e^{-\frac 32|\omega|}+e^{-\frac 32|\wt\omega|}.
\ee
\be\label{e:eq}
\left| \int Q^3(y)Q(y-\omega) dy - c_Q  I_Q  |\omega|^{-\frac 12}e^{-|\omega|}\right|
\lesssim |\omega|^{-\frac 32}e^{-|\omega|},
\ee
with $c_Q$ and $I_Q$ are given by~\eqref{as:q}--\eqref{d:Iq}.

\smallskip

\noindent{\it Proof of~\eqref{e:es}.} From~\eqref{as:q},
observe that
\begin{equation}\label{prod}
Q(y)Q(y-\omega)\lesssim 
(1+|y|)^{-\frac 12}(1+|y-\omega|)^{-\frac 12} e^{-|y|}e^{-|\omega|+|y|}
\lesssim |\omega|^{-\frac 12}  e^{-|\omega|}.\end{equation}
Thus,
\[
\int (1+|y|^q) Q^3(y)Q(y-\omega)dy\lesssim
 |\omega|^{-\frac 12}  e^{-|\omega|} \int  (1+|y|^q) Q^2(y) dy \lesssim  |\omega|^{-\frac 12}  e^{-|\omega|} .
\]

\noindent{\it Proof of~\eqref{e:esb}.} 
From  ~\eqref{prod},
\begin{align*}
\int  (1+|y|^q) Q^2(y)Q(y-\omega)Q(y-\wt\omega)dy  
& \lesssim \int  (1+|y|^q) Q^2(y)Q^{\frac 34}(y-\omega)Q^{\frac 34}(y-\wt\omega)dy\\
& \lesssim 
e^{-\frac 34 |\omega|}e^{-\frac 34 |\wt \omega|}
\int (1+|y|^q) Q^{\frac 12}(y) dy \lesssim
e^{-\frac 34 |\omega|}e^{-\frac 34 |\wt \omega|}.
\end{align*}
\noindent{\it Proof of~\eqref{e:eq}.} First,  using~\eqref{as:q},
\[
\int_{|y|>\frac 34|\omega|} Q^3(y) Q(y-\omega) dy \lesssim   e^{-\frac 94 |\omega|} \int Q(y-\omega) dy \lesssim  e^{-\frac 94 |\omega|}  .
\]
Second, for $|y|<\frac 34|\omega|$, we use~\eqref{as:q} to write
\[
\left|Q(y-\omega)-c_Q |y-\omega|^{-\frac 12} e^{-|y-\omega|}\right|
\lesssim |y-\omega|^{-\frac 32}e^{-|y-\omega|}
\lesssim |\omega|^{-\frac 32}e^{-|\omega|+|y|}.
\]
In particular,
\begin{align*}
&\left|\int_{|y|<\frac 34|\omega|} Q^3(y)Q(y-\omega) dy 
- c_Q \int_{|y|<\frac 34|\omega|} Q^3(y)|y-\omega|^{-\frac 12} e^{-|y-\omega|} dy\right|
 \lesssim |\omega|^{-\frac 32}e^{-|\omega|}.
\end{align*}
Still for $|y|<\frac 34|\omega|$, the expansion
\[
|y-\omega|^2 =  |\omega|^2 - 2 {y\cdot \omega} + {|y|^2}
\]
implies
\[
\left| |y-\omega|^{-\frac 12} - |\omega|^{-\frac 12}\right|
\lesssim |\omega|^{-\frac 32} |y|
\]
and
\[
\left| |y-\omega| - |\omega| + y\cdot \frac \omega{|\omega|} \right|\lesssim |\omega|^{-1} |y|^2 .
\]
Thus,
\[
\left| e^{-|y-\omega|} - e^{-|\omega|+ y\cdot\frac \omega{|\omega|}}\right|
\lesssim |\omega|^{-1} |y|^2 \left(e^{-|y-\omega|} +e^{-|\omega|+y\cdot\frac \omega{|\omega|}}\right)
\lesssim |\omega|^{-1} |y|^2 e^{-|\omega|}e^{|y|}.
\]
Therefore,
\[
\left| |y-\omega|^{-\frac 12} e^{-|y-\omega|} - |\omega|^{-\frac 12} e^{-|\omega|+ y\cdot\frac \omega{|\omega|}}\right|
\lesssim |\omega|^{-\frac 32} (1+|y|^2) e^{-|\omega|}e^{|y|},
\]
and so
\begin{align*}
&\left| \int_{|y|<\frac 34|\omega|} Q^3(y)|y-\omega|^{-\frac 12} e^{-|y-\omega|} dy
- |\omega|^{-\frac 12} e^{-|\omega|} \int_{|y|<\frac 34|\omega|} Q^3(y)  e^{y\cdot\frac \omega{|\omega|}} dy\right|
\\&\lesssim |\omega|^{-\frac 32} e^{-|\omega|} \int (1+|y|^2) e^{-2|y|}dy\lesssim |\omega|^{-\frac 32} e^{-|\omega|}.
\end{align*}
Also, we see that
\[
|\omega|^{-\frac 12} e^{-|\omega|}  \int_{|y|>\frac 34|\omega|} Q^3(y) e^{y\cdot\frac \omega{|\omega|}} dy
\lesssim |\omega|^{-\frac 12} e^{-|\omega|}  \int_{|y|>\frac 34|\omega|} e^{-2|y|} dy
\lesssim e^{-2|\omega|}.
\]
Since for all $\omega\neq 0$ (see~\eqref{d:Iq}),
\[
I_Q = \int Q^3(y) e^{y\cdot\frac \omega{|\omega|}} dy,
\]
we have proved~\eqref{e:eq}.

\smallskip

\noindent\textbf{step 4.} Estimates of $G_{k}$ and $\Psi_{Q_a}$. We are now in position to prove~\eqref{b:Gk} and~\eqref{b:iGk}. The estimate on $\Psi_{Q_a}$ in~\eqref{b:Gk}  follows directly from its definition~\eqref{d:PQa}. To estimate $G_k$ as in~\eqref{b:Gk}, we first note that from~\eqref{as:q},~\eqref{d:rh} and the definition of $a(z)$ in~\eqref{d:a}, we have, for some $q>0$,
\[
|Q_a|\lesssim |y|^{-\frac 12} e^{-|y|} + |a(z)| |y|^{q}  e^{-|y|}
\lesssim (1+|y|)^{-\frac 12} e^{-|y|} +  z^{\frac 12} e^{-\kappa z}(1+|y|)^{q}e^{-|y|}.
\]
Moreover, for $j\neq k$, from the definition of $\kappa$ in~\eqref{d:ka},
\[
|z_j-z_k| = z |{\rm e}_k-{\rm e}_j|  \geq   \kappa z.
\]
From this, it follows easily that for $j\neq k$,
\[
|Q_a(y)||Q_a(y-(z_k-z_j))|\lesssim z^{-\frac 12} e^{-\kappa z},
\]
which in light of the explicit formula~\eqref{d:Gk} yields the control of $G_k$ in~\eqref{b:Gk}.\\
We now turn to the proof of ~\eqref{b:iGk} which requires a   more careful analysis of the interaction terms. We first compute the main order of the contribution of $G_1^{\rm (I)}$ to $\la G_1,iQ\ra$.
For $j=2,\ldots,K$,
\begin{align*}
&\la \big[e^{-i\Gamma_1}Q_a^2\big] (y) \big[e^{i\Gamma_j}Q_a\big](y-(z_j-z_1)),iQ_a\rangle\\
& = \int Q_a^3(y)  Q_a(y-z(\mathrm{e}_j-\mathrm{e}_1)) \sin(\Gamma_j(y-z(\mathrm{e}_j-\mathrm{e}_1))-\Gamma_1(y)) dy.
\end{align*}
First,  by the decay of $\rho$ (see~\eqref{d:rh}),  ~\eqref{e:es} and the definition of $a(z)$ in~\eqref{d:a}, we have
\begin{align*}
& \int \left| Q_a^3(y)  Q_a(y-z(\mathrm{e}_j-\mathrm{e}_1)) -  Q^3(y) Q(y-z(\mathrm{e}_j-\mathrm{e}_1))\right| dy
\lesssim |a| z^{\frac 52} e^{-\kappa z}
\lesssim z^3  e^{-2 \kappa z}.
\end{align*}
Next, note that, since $\Gamma_j= \beta_j \cdot y - \frac b4|y|^2$,
\begin{align*}
&\left|\sin(\Gamma_j(y-z(\mathrm{e}_j-\mathrm{e}_1))-\Gamma_1(y))-(\Gamma_j(y-z(\mathrm{e}_j-\mathrm{e}_1))-\Gamma_1(y))\right| 
\\ &\lesssim |\Gamma_j(y-z(\mathrm{e}_j-\mathrm{e}_1))|^2+|\Gamma_1(y)|^2 \lesssim |\beta|^2 (|y|^2+z^2)+|b|^2(|y|^4+z^4),
\end{align*}
and 
\[\left| (\Gamma_j(y-z(\mathrm{e}_j-\mathrm{e}_1))-\Gamma_1(y)) + \frac{b}4 |z(\mathrm{e}_j-\mathrm{e}_1)|^2 \right|
\lesssim |\beta| (|y|+z) + |b| (|y|^2 + |y|z).\]
Thus, using~\eqref{e:es},
\begin{align*}
& \int Q^3(y) Q(y-z(\mathrm{e}_j-\mathrm{e}_1)) \Big|\sin(\Gamma_j(y-z(\mathrm{e}_j-\mathrm{e}_1))-\Gamma_1(y)) + \frac{b}4 |z(\mathrm{e}_j-\mathrm{e}_1)|^2 \Big| dy\\
& \lesssim (|\beta|^2 z^2+|b|^2z^4 + |\beta| z+|b|z) z^{-\frac 12} e^{-\kappa z}.
\end{align*}
Therefore, we have proved
\begin{align*}
&\left|\la \big[e^{i\Gamma_1} Q_a^2\big] (y) \big[e^{i\Gamma_j} Q_a\big](y-z(\mathrm{e}_j-\mathrm{e}_1)),iQ_a\rangle 
+  \frac{b}4 |z(\mathrm{e}_j-\mathrm{e}_1)|^2  \int Q^3(y) Q(y-z(\mathrm{e}_j-\mathrm{e}_1))\right|\\
& \lesssim (|\beta|^2 z^2+|b|^2z^4 + |\beta| z+|b|z) z^{-\frac 12} e^{-\kappa z}+z^3 e^{-2 \kappa z}.
\end{align*}
For $j=2$ and $j=K$, we have $|z(\mathrm{e}_j-\mathrm{e}_1))|=\kappa z$, and so using~\eqref{e:eq},
\begin{align*}
&\left|\la  \big[e^{i\Gamma_1} Q_a^2\big] (y) \big[e^{i\Gamma_j} Q_a\big](y-z(\mathrm{e}_j-\mathrm{e}_1)),iQ_a\rangle
+ \frac{b}4c_Q I_Q \kappa^{\frac 32}  z^{\frac 32} e^{-\kappa z}  \right|\\
& \lesssim (|\beta|^2 z^2+|b|^2z^4 + |\beta| z+|b|z) z^{-\frac 12} e^{-\kappa z}
+z^3 e^{-2 \kappa z}.
\end{align*}
For $K\geq 4$ and $j=3,\ldots,K-1$, we have 
$|\mathrm{e}_j-\mathrm{e}_1|> \kappa'$, for some $\kappa'>\kappa$.
Thus the following bound follows from similar computations  
\begin{align*}
&\left|\la \big[e^{i\Gamma_1} Q_a^2\big] (y) \big[e^{i\Gamma_j} Q_a\big](y-z(\mathrm{e}_j-\mathrm{e}_1)),iQ_a\rangle\right| \\
&\lesssim 
(|\beta|^2 z^2+|b|^2z^4 + |\beta| z+|b|z^2) z^{-\frac 12} e^{-\kappa' z} +z^3 e^{-2 \kappa z}.
\end{align*}
Note that $\la 2 G_1^{\rm (I)}  + \overline G_1^{\rm (I)} , iQ_a \ra = \la G_1^{\rm (I)}, iQ_a\ra$. We finally bound the contribution of $G_1^{\rm (II)}$. For $j\neq 1,$ $l\neq 1$ and $l\neq j$,
\begin{align*}
& \la \left[ e^{-2i \Gamma_k} Q_a\right](y)  \left[e^{i \Gamma_j} Q_a\right](y-(z_j-z_k)) \left[e^{i\Gamma_l} Q_a\right](y-(z_l-z_k)),iQ_a\ra
\\&= \int Q_a^2(y)  Q_a(y-(z_j-z_k)) Q_a(y-(z_l-z_k))\\
& \qquad \times\sin(\Gamma_j(y-(z_j-z_k))+\Gamma_l(y-(z_l-z_k)) -2 \Gamma_k(y)) dy.
\end{align*}
By~\eqref{e:esb}, the bound on $|a|$ and
$|\Gamma_j|\lesssim |\beta| (|y|+z) + |b| (|y|^2 + z^2)$,
this term is bounded by $(|\beta| z+|b|z^2) e^{-\frac 32 \kappa z}$.\\
Gathering these estimates, using the definition of the constant $c_a$ in~\eqref{d:a} which takes into account the two different cases $K=2$ and $K\geq 2$ 
(for $K=2$, the soliton $P_1$ has nonlinear interaction with only one other soliton, while for $K\geq 3$, it has  exactly two closest neighboring solitons, $P_2$ and $P_K$),
 we obtain finally
\begin{align*}
|\la G_k,iQ_a\ra + \kappa c_a \la \rho,Q\ra b z^{\frac 32} e^{-\kappa z} | \lesssim
(|\beta|^2 z^2+|b|^2z^4 + |\beta| z+|b|z) z^{-\frac 12} e^{-\kappa z}
+z^3 e^{-2\kappa z},
\end{align*}
which completes the proof of~\eqref{b:iGk}.
\end{proof}


\subsection{Formal resolution of the modulation system with forcing}\label{s:2.2}


From Lemma~\ref{pr:P}, we derive a simplified modulation system with forcing term and we determine one of its   approximate solution that is relevant for the regime of Theorem~\ref{thmmain}.
Moreover, we justify the special choice of function $a(z)$ in~\eqref{d:a}.
Formally, i.e. assuming that $\PP$ is a solution of~\eqref{nls} up to error terms of lower order than the ones in~\eqref{d:psik} (making this rigorous will be the object of the bootstrap estimates in Sect.~\ref{s:4}), we have the following bounds  ($\vec m^a_k$ is defined in~\eqref{d:mkM})
\begin{equation}\label{formal1}
|\vec m_1^a|\lesssim z^{-\frac 12} e^{-\kappa z}.
\end{equation}
Indeed,~\eqref{formal1} is obtained from~\eqref{d:psik}--\eqref{b:Gk} by projecting $\EE_\PP$ onto   directions 
related to the generalized null space~\eqref{e:s} (see  Lemma~\ref{l:m} for rigorous computations).
To simplify the discussion, we drop the equation of $\gamma$, which is not coupled with any other equation and has no influence on the regime. Next, we see that using the first line of $\vec m^a_1$, i.e. $|b+\lsl|\lesssim z^{-\frac 12} e^{-\kappa z}$, we can replace $\lsl$ by $-b$ in all the other estimates.
Similarly, we insert the estimate on $\dot z$ from the second line into the estimate for $\dot \beta$.
We obtain the following simplified system
\begin{equation}\label{formal3}
|b+\lsl|+
|\dot z  - 2\beta  - b z| +  |\dot \beta + b \beta | + |\dot b + b^2  -a|\lesssim   z^{-\frac 12} e^{-\kappa z}.
\end{equation}
It is easy to check the following estimates 
\begin{lemma}\label{formal4}
Let $(z_{\rm app},\lambda_{\rm app}(s),\beta_{\rm app},b_{\rm app}(s))$ be such that
\begin{equation}\label{formal5}
\begin{aligned}
& \lambda_{\rm app}(s) = \log^{-1}(s),\quad
z_{\rm app}^{-\frac 32}(s) e^{\kappa z_{\rm app}(s)}
= \frac  {\kappa c_a}2 s^2,\\
& |\beta_{\rm app}(s) |\lesssim  s^{-1} \log^{-\frac 32}(s),\quad
b_{\rm app}(s) = s^{-1} \log^{-1}(s).
\end{aligned}
\end{equation}
Then,
\begin{equation}\label{formal6}
\begin{aligned}
& z_{\rm app}(s) \sim \frac 2\kappa \log(s),\quad
|b_{\rm app}+\frac{\dot \lambda_{\rm app}}{\lambda_{\rm app}}|=0,\quad
|\dot z_{\rm app}  - 2\beta_{\rm app}  - b_{\rm app} z_{\rm app}| \lesssim s^{-1} \log^{-\frac 12}(s),\\
& |a(z_{app}) + s^{-2} \log^{-1}(s) |\lesssim s^{-2}  \log^{-\frac 32}(s),
\quad  |\dot b_{\rm app} + b_{\rm app}^2  -a(z_{\rm app})|\lesssim s^{-2} \log^{-\frac 32}(s).
\end{aligned}\end{equation}
\end{lemma}
The above estimates mean that~\eqref{formal5} is a reasonnable guess for the first order asymptotics as $s\to +\infty$ of some particular solutions of
\eqref{formal3} (we refer to Sect.~\ref{:s:4.4} for a rigorous integration of~\eqref{formal3}).
Note that we do not actually determine the main order of $\beta(s)$; to do this, more interaction computations would be necessary. However, since $|\dot \beta + b \beta |\lesssim z^{-\frac 12} e^{-\kappa z}$, formally, we obtain
$|\dot \beta|\lesssim s^{-2} \log^{-2}(s)$, which justifies a bootstrap on $\beta(s)$ of the form
$|\beta(s)|\ll s^{-1} \log^{-2}(s)$. 
Note also that there exist   solutions of~\eqref{formal5} with different asymptotics,
 corresponding to (NLS) solutions like $v(t)$ of Corollary~\ref{maincor}.

\smallskip

To complete this formal discussion,
we justify the choice of $a(z)$ in~\eqref{d:a} in the regime given by~\eqref{formal5}.
Indeed, projecting $\Psi_1$ onto the direction $iQ_a$, from~\eqref{b:iGk}, we obtain at the leading order
\begin{equation}\label{formal2}
|\dot z a'(z) - \kappa c_a b z^{\frac 32} e^{-\kappa z}|\lesssim
(|\beta|^2 z^2+|b|^2z^4 + |\beta| z+|b|z) z^{-\frac 12} e^{-\kappa z} + z^3 e^{-2\kappa z}. 
\end{equation}
In the regime suggested by~\eqref{formal5}, since $|\beta|\ll |b|z$, we have $bz\sim \dot z$ and thus,
simplifying $\dot z$, we obtain
\[
|  a'(z) - \kappa c_a   z^{\frac 12} e^{-\kappa z}| \lesssim z^{-\frac 12} e^{-\kappa z},
\]
which justifies the definition~\eqref{d:a} by integrating in $z$.


\subsection{Modulation of the approximate solution}

 
We state a standard modulation result around $\PP$. 
We  restrict ourselves to the case of solutions invariant by the  rotation preserving $\PP$.
Denote by $\tau_K$ the rotation of center $0$ and angle $\frac{2\pi}{K}$ in $\RR^2$. Since $Q$ and $\rho$ are radial, by definition of $P_k$ and $\beta_k$, $z_k$ in~\eqref{d:Pk} and~\eqref{d:bzk}, we have
for $k\in \{1,\ldots, K-1\}$,  $P_k( y) = P_{k+1}(\tau_K  y)$ and $P_K(y) = P_1(\tau_K y)$.
In particular, it follows that $\PP(\tau_K y)=\PP(y)$, i.e. $\PP$ is invariant by the rotation $\tau_K$.
Note also that equation~\eqref{nls} is invariant by rotation. In particular, if a solution of~\eqref{nls} is invariant by the rotation $\tau_K$ at some time, then it is invariant by rotation at any time. In this context, the  following modulation result relies on a standard argument based on the Implicit Function Theorem (see e.g. Lemma~2 in~\cite{MeRa05}) and we omit its proof.

\begin{lemma}[Modulation around $\PP$]
\label{le:mod}
Let $I$ be some time interval. Let $u \in C(I,H^1(\RR^2))$ be a solution of~\eqref{nls} 
  invariant by the rotation $\tau_K$ and
such that
\be\label{close}
\sup_{t\in I} 
\left\| e^{-i\tilde\gamma (t)}\tilde\lambda (t) u(t,\tilde\lambda (t) \,.\,)- \sum_k Q(\,.\,-{\rm e}_k \tilde z (t))\right\|_{H^1}<\delta
\ee
for some $\tilde\lambda(t)>0$, $\tilde\gamma(t)\in \RR^2$, $\tilde z(t)>\delta^{-1}$, where $\delta>0$ is small enough.
Then, there exist a $C^1$ function
\[
\vec p=(\lambda,z,\gamma,\beta,b):I\to (0,\infty)^2\times\RR^3,
\]
such that,  for $\PP(t,y)=  \PP(y;z(t),b(t),\beta(t))$ as defined in~\eqref{d:PP},
the solution $u(t)$ decomposes on $I$   as
\be\label{dec}
u(t,x) = \frac {e^{i\gamma(t)}} {\lambda(t)} (\PP+\e)(t,y),\quad y=\frac{x}{\lambda(t)},
\ee
where for all $t\in I$,
\be\label{small}
|b(t)|+|\beta(t)|+\|\e(t)\|_{H^1}+|z(t)|^{-1}\lesssim \delta,
\ee
and, setting $\e(t,y)  = \big[e^{ i\Gamma_1} \eta_1\big](t,y-z_1)$,
\be\label{ortho}
\la \eta_1(t),|y|^2Q \ra =|\la \eta_1(t),yQ \ra|=
\la \eta_1(t),i\rho\ra=|\la\eta_1(t),i\nabla Q\ra|+ \la \eta_1(t),i \Lambda Q\ra=0.
\ee
Moreover, $\e $ is also invariant by the rotation $\tau_K$.
\end{lemma}
Note that the choice of the special orthogonality conditions~\eqref{ortho} is related to the generalized null space of the linearized equation around $Q$,~\eqref{e:s} and to the coercivity property~\eqref{e:co}. See the proof of Lemma~\ref{l:m} for a technical justification of these choices (see also \cite{RaSz11}).

\section{Backwards uniform estimates}\label{s:3}
In this section, we  prove uniform estimates on particular backwards solutions.
The key point is to   carefully adjust their final data to obtain uniform estimates corresponding to the special  regime of Theorem~\ref{thmmain} and Lemma~\ref{formal4}.

\smallskip

Let $(\lambda^{in},z^{in},b^{in})\in  (0,+\infty)^2\times \RR$ to be chosen
with $\lambda^{in}\ll 1$, $z^{in}\gg 1$, $|b^{in}|\ll 1$. Let
$u(t)$ for $t\leq 0$ be the solution of~\eqref{nls} with data (see~\eqref{d:PP})
\be\label{init}
u(0,x)= \frac{1}{\lambda^{in}}\PP^{in}\left(\frac{x}{\lambda^{in}}\right)
\quad \hbox{where}\quad 
\PP^{in}(y)= \PP(y;(z^{in},b^{in},0))
\ee
(we   arbitrarily fix $\gamma^{in}=\beta^{in}=0$).
Note that $u(0)$ satisfies~\eqref{close} and, by continuity of the solution of~\eqref{nls} in $H^1$, it exists and satisfies~\eqref{close}  
on some maximal time interval $(t^{mod},0]$, where $t^{mod}\in [-\infty,0)$. 
Note also that by invariance by rotation of  equation~\eqref{nls}, $u(t)$ is invariant by the rotation $\tau_K$.
On $(t^{mod},0]$, we consider $(\vec p,\e)$  the decomposition of $u$
defined from Lemma~\ref{le:mod}. For   $s^{in}\gg 1$, we normalize the rescaled time $s$ as follows,
for $t\in (t^{mod},0]$,
\be\label{d:s}
s=s(t)=s^{in} - \int_t^{0} \frac{d\tau}{\lambda^2(\tau)}.
\ee
Observe from~\eqref{init} that
\be\label{c:in2}
\begin{aligned}
&\lambda(s^{in})=\lambda^{in},\quad  b(s^{in})=b^{in},\quad z(s^{in})=z^{in},\\
&\gamma(s^{in})=0,\quad \beta(s^{in})=0,\quad\e(s^{in})\equiv 0.
\end{aligned}\ee 
\begin{proposition}[Uniform backwards estimates]
\label{p:st}
There exists $s_0>10$ such that for all $s^{in}> s_0$, there exists a choice of  parameters
$(\lambda^{in},z^{in},b^{in})$ with
\be\label{c:in}
\begin{aligned}
&\left|  \left(\frac{2}{\kappa c_a}\right)^{\frac 12} (z^{in})^{-\frac 34} e^{\frac \kappa 2 z^{in}} - s^{in}\right|  < s^{in}\log^{-\frac 12} (s^{in}),\\
  &\lambda^{in} = \log^{-1}(s^{in}),\quad b^{in} =  \left(\frac{2c_a}{\kappa}\right)^{\frac 12} (z^{in})^{-\frac 14} e^{-\frac\kappa 2 z^{in}},  
\end{aligned}\ee
such that the solution $u$ of~\eqref{nls} corresponding to~\eqref{init} exists and satisfies~\eqref{close} on the rescaled interval of time $[s_0,s^{in}]$, the rescaled time $s$ being defined  in~\eqref{d:s}.
Moreover, the decomposition of $u$ given by Lemma~\ref{le:mod} on $[s_0,s^{in}]$ \[
u(s,x) = \frac {e^{i\gamma(s)}} {\lambda(s)} (\PP+\e)(s,y),\quad y=\frac{x}{\lambda(s)},
\]
satisfies the following uniform estimates, for all  $s\in [s_0,s^{in}]$,
\be\label{b:su}
\begin{aligned} 
&\left|z(s)- \frac{2}{\kappa}\log(s)\right|\lesssim \log(\log(s)),  \quad
\left| \lambda(s) - \log^{-1}(s)\right| \lesssim \log^{-\frac 32} (s),
\\
& \left|b(s) -  s^{-1} \log^{-1}(s)\right| +|\beta(s)|+\|\e(s)\|_{H^1}\lesssim s^{-1} \log^{-\frac 32}(s),\quad
|a(s)|\lesssim s^{-2}\log^{-1} (s).
\end{aligned}
\ee
\end{proposition}
The key point in Proposition~\ref{p:st} is that $s_0$ and the constants in~\eqref{b:su} are independent of $s^{in}$ as $s^{in}\to +\infty$. 
Observe that estimates~\eqref{b:su} match the discussion of Sect.~\ref{s:2.2}.

\smallskip

The rest of this section is devoted to the proof of Proposition~\ref{p:st}. The proof relies on a bootstrap argument, integration of the  differential system of geometrical parameters  and  energy estimates.
We estimate $\e$ by standard energy arguments in the framework of multi-bubble solutions.
The particular regime of the geometrical parameters does not create any further difficulty. On the contrary, 
the special behavior  $b(s)\sim s^{-1}\log^{-1}(s)$ simplifies this part of the proof (see step 2 of the proof of Proposition~\ref{p:coe}).
We control the geometrical parameters of the bubbles in the bootstrap regime  adjusting the final data $(\lambda^{in},z^{in},b^{in})$.

\subsection{Bootstrap bounds}\label{s:34} The proof of Proposition~\ref{p:st} follows from 
bootstrapping the following estimates, chosen in view of the formal computations in Sect.~\ref{s:2.2},
\begin{equation}\label{BS}
\begin{aligned}
&  \left| \left(\frac{2}{\kappa c_a}\right)^{\frac 12} z^{-\frac 34} e^{\frac \kappa 2 z} - s\right|  \leq s\log^{-\frac 12} (s),
\\&   \frac 12 s^{-1} \log^{-1}(s) \leq  b(s) \leq 2  s^{-1} \log^{-1} (s),
\\ & |\beta(s)|\leq s^{-1} \log^{-\frac 3 2}(s) ,\quad  \|\e(s)\|_{H^1} \leq   s^{-1} \log^{-\frac 3 2}(s).
\end{aligned}
\end{equation}
Note that the estimate on $z$ in~\eqref{BS} immediately implies that, for $s$ large 
\be\label{BS2}
e^{-\kappa z} \lesssim s^{-2} \log^{-\frac 32}(s),\quad
\left| z(s) -\frac{2}{\kappa} \log(s)\right| \lesssim \log(\log(s)),\quad
|a(s)| \lesssim s^{-2} \log^{-1}(s).
\ee
For $s_0>10$ to be chosen large enough (independent of $s^{in}$), and all $s^{in} \gg s_0$, we define
\be\label{d:ss}
s^* = \inf \{ \tau \in [s_0,s^{in}];\ \hbox{\eqref{BS} holds on $[\tau,s^{in}]$}\}.
\ee

\subsection{Control of the modulation equations}
We claim the following bounds on the modulation system $\vec m^a_1$ and on the error $\EE_\PP$ given by~\eqref{d:mkM},~\eqref{d:EE}--\eqref{d:psik} in the bootstrap regime~\eqref{BS}.

\begin{lemma}[Pointwise control of the modulation equations and the error]
\label{l:m}
The following estimates hold on $[s^{*},s^{in}]$.
\be\label{b:m}
|\vec m^a_1(s)|\lesssim s^{-2} \log^{-2}(s).
\ee
\be\label{b:os}
|\la \eta_1(s), Q\ra| \lesssim s^{-2} \log^{-2}(s),
\ee
\be\label{b:ode1}
|\dot z -b z| \lesssim s^{-1} \log^{-1}(s), \quad |\dot \beta|+|\dot b   -a|\lesssim s^{-2} \log^{-2}(s).
\ee
Moreover,  for all $s\in [s^*,s^{in}]$, for all $y\in \RR^2$,
\be\label{b:EE}
 {|\EE_\PP(s,y)|+|\nabla\EE_\PP(s,y)|}
\lesssim s^{-2} \log^{-2}(s)  \,  {\sum_k Q^{1/2}(y-z_k(s))}   .
\ee
\end{lemma}
\begin{proof}[Proof of Lemma~\ref{l:m}] The proofs of the first two estimates are to be combined. Since $\e(s^{in})\equiv 0$, we may define
\[
s^{**} = \inf\{s\in [s^*,s^{in}];\ |\la \eta_1(\tau), Q\ra| \leq C^{**} \tau^{-2} \log^{-2}(\tau) \hbox{ holds on $[s,s^{in}]$}\},
\]
for some constant $C^{**}>0$ to be chosen large enough. We work on the interval $[s^{**},s^{in}]$.

\smallskip

\noindent\textbf{step 1} Equation of $\e$ and change of variable. Let $v=\PP+\e$ in~\eqref{d:v}. It follows from~\eqref{e:v},~\eqref{d:EE} that
\be\label{e:e}
i\dot \e + \Delta \e - \e + \left( |\PP+\e|^2(\PP+\e)-|\PP|^2 \PP\right)
-i\lsl \Lambda \e + (1-\dot \gamma)\e +\EE_{\PP}=0.
\ee
By rotation symmetry (see Lemma~\ref{le:mod}) it is enough to understand the orthogonality for $\eta_1$.
Thus, we change the space variable to match the one of the bubble $P_1$.
Recall that we have defined
$
\e(s,y) = \big[e^{i \Gamma_1} \eta_1\big](s,y-z_1).
$
Denote
\[
\PP(s,y)= \big[e^{i \Gamma_1} \PP_1\big](s,y-z_1),\quad
\EE_{\PP}(s,y)=\big[ e^{i \Gamma_1}\EE_{\PP_1}\big](s,y-z_1).
\]
We  rewrite the equation of $\e$ into the following equation for $\eta_1$
(see also step 1 of the proof of Lemma~\ref{pr:P})
\be\label{e:eta}
i\dot \eta_1 + \Delta \eta_1 - \eta_1 + (|\PP_1+\eta_1|^2(\PP_1+\eta_1) - |\PP_1|^2 \PP_1)   
+ \vec m^0_1\cdot \vec \md \eta_1 +\EE_{\PP_1}=0,
\ee
Recall also that $\eta_1$ satisfies the orthogonality conditions~\eqref{ortho}.

\smallskip

\noindent\textbf{step 2} General null space like computation.
Let $A(y)$ and $B(y)$ be two real-valued functions in $\mathcal Y$.
We claim the following estimate on $[s^{**},s^{in}]$
\be\label{o:gen}
\left| \frac{d}{ds}\la \eta_1 , A+iB\ra  
- \left[ \la \eta_1, i L_- A - L_+ B \ra - \la \vec m^a_1 \cdot \vec\md Q, i A- B\ra \right]  \right|
\lesssim s^{-2} \log^{-2}(s)+s^{-1} |\vec m^a_1|.
\ee
We compute from~\eqref{e:eta},
\[\label{etab}\begin{split}
\frac{d}{ds}\la \eta_1 , A+iB\ra &= \la \dot\eta_1 , A+iB\ra =\la  i\dot \eta_1 , iA-B\ra\\
& = \la -\Delta \eta_1 + \eta_1 - (2 Q^2 \eta_1 + Q \overline\eta_1), iA-B\ra\\
& -\la |\PP_1+\eta_1|^2 (\PP_1+\eta_1) - |\PP_1|^2\PP_1 -2 Q^2 \eta_1 -Q^2 \overline\eta_1 ,iA-B\ra \\&
- \la \vec m^0_1\cdot \vec\md\eta_1,iA-B \ra  - \la \EE_{\PP_1} ,iA-B\ra.
\end{split}\]
First, since $A$ and $B$ are real-valued, we have
\[
\la -\Delta \eta_1 + \eta_1 - (2 Q^2 \eta_1 + Q \overline\eta_1), iA-B\ra 
=\la \eta_1, i L_- A - L_+ B \ra.
\]
Second, note that
\begin{align*}
&|\PP_1+\eta_1|^2 (\PP_1+\eta_1) - |\PP_1|^2\PP_1-2 Q^2 \eta_1 -Q^2 \overline\eta_1 \\
& = 2 (|\PP_1|^2-Q^2) \eta_1 + (\PP_1^2-Q^2) \overline\eta_1 
+2\PP_1 |\eta_1|^2 +\overline \PP_1 \eta_1^2 + 2 \PP_1|\eta_1|^2 + |\eta_1|^2 \eta_1,
\end{align*}
and recall the expression of $\PP_1$
\[
\PP_1= Q + a\rho + \sum_{k=2}^K e^{i(\Gamma_k(y-(z_k-z_1))-\Gamma_1(y))} Q_a(y-(z_k-z_1)).
\]
Therefore,  using $A,B\in \mathcal Y$,~\eqref{BS}--\eqref{BS2} and $|z_k-z_j|\geq \kappa z$, for $k\neq j$,
 we have, for some $q>0$,
\begin{align*}
&|\la (|\PP_1|^2-Q^2)\eta_1,iA-B\ra|+|\la (\PP_1^2 - Q^2)\overline \eta_1,iA-B\ra|
\\&\lesssim (|a|+z^q e^{-\kappa z}) \|\eta_1\|_{L^2}
\lesssim   s^{-3} \log^{q}(s).
\end{align*}
Next,
\[
|\la 2\PP_1 |\eta_1|^2 +\overline \PP_1 \eta_1^2 + 2 \PP_1|\eta_1|^2 ,iA-B\ra|
\lesssim \|\e\|_{L^2}^2 \lesssim  s^{-2}\log^{-3}(s),
\]
\[
|\la |\eta_1|^2 \eta_1, iA-B\ra| \lesssim  \la |\e|^3  , |A|+|B|\ra 
\lesssim \|\e\|_{H^1}^3 \lesssim  s^{-3} \log^{-\frac 92}(s) .
\]
Third, we have, using~\eqref{BS}--\eqref{BS2}, integration by parts,
\begin{align*}
|\la \vec m^0_1\cdot \vec\md\eta_1,iA-B \ra|
&\lesssim \left|\la \vec m^a_1\cdot \vec\md\eta_1,iA-B \ra\right|+\left|\la (\vec m^a_1-\vec m^0_1)\cdot \vec\md\eta_1,iA-B \ra\right|\\
&\lesssim s^{-1} \log^{-\frac 32}(s) |\vec m^a_1| +s^{-3} \log^{-\frac 52}(s).
\end{align*}
Finally, we claim the following estimate, which is enough to complete the proof of~\eqref{o:gen}.
\be\label{o:gen2}
\left| \la \EE_{\PP_1} , iA-B \ra - \la \vec m_1^a \cdot \vec \md Q,iA-B\ra \right|
\lesssim s^{-2} \log^{-2}(s) + s^{-1} |\vec m_1^a|.
\ee
Indeed, recall the expression of $\EE_{\PP_1}$ (from~\eqref{d:EE}--\eqref{d:psik})
\begin{align*}
\EE_{\PP_1} 
&= \Psi_1 + \sum_{k=2}^K e^{i(\Gamma_k(y-(z_k-z_1))-\Gamma_1(y))} \Psi_k(y-(z_k-z_1)).
\\& = \vec m_1^a \cdot \vec \md Q_a +i \dot  z a'(z) \rho + G_1 +\Psi_{Q_a}
+\sum_{k=2}^K e^{i(\Gamma_k(y-(z_k-z_1))-\Gamma_1(y))} \Psi_k(y-(z_k-z_1)). 
\end{align*}
First, since $Q_a= Q+a \rho$, by~\eqref{BS}--\eqref{BS2},
\[
\left| \la \vec m_1^a \cdot \vec \md (Q_a-Q),iA-B\ra\right|
\lesssim |a| |\vec m_1^a| \lesssim s^{-2} \log^{-1}(s) |\vec m_1^a|.
\]
Second, from~\eqref{BS}--\eqref{BS2},
\be\label{pz}\begin{split}
|\la \dot  z a'(z) \rho  ,iA-B\ra| 
&\lesssim  |a'(z)| |\dot z|
\lesssim s^{-2}\log^{-1}(s) ( |\vec m_1^a| + |\beta| + |\lsl| z) \\
&\lesssim s^{-2}\log^{-1}(s) ((z+1) |\vec m_1^a| +|\beta|+ |b| z )\\
&\lesssim s^{-2} |\vec m_1^a|+ s^{-3}\log^{-1}(s).
\end{split}\ee
Third, from~\eqref{b:Gk} and~\eqref{BS}--\eqref{BS2},
\[
|\la G_1, iA-B\ra|\lesssim \|G_1\|_{L^\infty} \lesssim z^{-\frac 12} e^{-\kappa z} \lesssim s^{-2} \log^{-2}(s).
\]
Fourth, from~\eqref{b:Gk} and~\eqref{BS}--\eqref{BS2},
\[
|\la \Psi_{Q_a}, iA-B\ra|\lesssim \|\Psi_{Q_a}\|_{L^\infty} \lesssim |a|^2 \lesssim s^{-4} \log^{-2}(s).
\]
Last, since $A,B\in \mathcal Y$,  for $k\geq 2$, we have
\[
|\la e^{i(\Gamma_k(y-(z_k-z_1))-\Gamma_1(y))} (\vec m_k^a \cdot \vec \md Q_a(.-(z_k-z_1))), iA-B\ra|
\lesssim s^{-1} |\vec m_k^a|,
\] 
and, proceeding as before for the other terms in $\Psi_k$, we obtain
\[
|\la e^{i(\Gamma_k(y-(z_k-z_1))-\Gamma_1(y))} \Psi_k(y-(z_k-z_1)) , iA-B\ra|\lesssim s^{-1} |\vec m_k^a| + s^{-2} \log^{-2}(s).
\] 
The proof of~\eqref{o:gen2} is complete.

\smallskip

\noindent{\textbf{step 3}} Modulation equations. We now use~\eqref{ortho} and~\eqref{o:gen}  to control the modulation vector $\vec m^a_1$.
Using~\eqref{o:gen}, we   draw one by one  the consequences of the orthogonality relations~\eqref{ortho}.
Note that the special orthogonality conditions~\eqref{ortho}, related to cancellations  (see~\eqref{e:s}) are crucial in these computations.

\smallskip

\noindent $\underline{\la \eta_1,|y|^2 Q \ra = 0}$. Let $A=|y|^2 Q$ and $B=0$. 
Since $L_- (|y|^2 Q) = -4 \Lambda Q$, $\la \eta_1, i \Lambda Q\ra =0$ and
$\la \vec m^a_1\cdot \vec \md Q, i |y|^2 Q\ra = - (b+\lsl) \la \Lambda Q, |y|^2 Q\ra = c_1 (b+\lsl)$, where $c_1\neq 0$, we obtain
\be\label{md1}
\Big| b + \lsl \Big| \lesssim s^{-2} \log^{-2}(s)+s^{-1} |\vec m^a_1|. 
\ee

\smallskip

\noindent $\underline{\la \eta_1, yQ \ra = 0}$. Let $A=y Q$ and $B=0$. 
Since $L_- (y Q) = -2 \nabla  Q$, $\la \eta_1, i \nabla Q\ra =0$ and
$\la \vec m^a_1\cdot \vec \md Q, i y Q\ra = - (\dot z -2 \beta + \lsl z) \la \nabla Q, y Q\ra = c_2 (\dot z -2 \beta + \lsl z)$, where $c_2\neq 0$, we obtain
\be\label{md2}
\Big| \dot z -2 \beta + \lsl z  \Big| \lesssim s^{-2} \log^{-2}(s)+s^{-1} |\vec m^a_1|. 
\ee

\smallskip

\noindent $\underline{\la \eta_1, i \rho \ra = 0}$. Let $A=0$ and $B=\rho$. 
Since $L_+ \rho = |y|^2  Q$, $\la \eta_1, |y|^2 Q\ra =0$ and
\[
\la \vec m^a_1\cdot \vec \md Q, |y|^2 Q \ra = c_3 (\dot \gamma - 1 + \beta^2- \lsl \beta z-\beta \dot z)
+c_4 (\dot b + b^2 - 2 b(b+\lsl) - a),   
\]
where $c_3,c_4\neq 0$, we obtain, for some $c$,
\be\label{md3}
\Big|  (\dot \gamma - 1 + \beta^2- \lsl \beta z-\beta \dot z)
+c(\dot b + b^2 - 2 b(b+\lsl) - a) \Big| \lesssim s^{-2} \log^{-2}(s)+s^{-1} |\vec m^a_1|. 
\ee

\smallskip

\noindent $\underline{\la \eta_1, i\nabla Q \ra = 0}$. Let $A=0$ and $B=\nabla Q$. 
Since $L_+ \nabla Q = 0$, and 
$\la \vec m^a_1\cdot \vec \md Q,\nabla Q\ra = c_5(\dot \beta -\lsl \beta+ \frac b2(\dot z- 2\beta + \lsl z )$, where $c_5\neq 0$, we obtain
\be\label{md4}
\Big|  \dot \beta -\lsl \beta  + \frac b2(\dot z- 2\beta + \lsl z ) \Big| \lesssim s^{-2} \log^{-2}(s)+s^{-1} |\vec m^a_1|. 
\ee

\smallskip

\noindent $\underline{\la \eta_1, i\Lambda Q \ra = 0}$. Let $A=0$ and $B=\Lambda Q$. 
Note that $L_+ (\Lambda Q) = -2Q$, and by the definition of $s^{**}$, $|\la \eta_1, Q \ra| \lesssim C^{**} s^{-2} \log^{-2}(s)$. Moreover,
$\la \vec m^a_1\cdot \vec \md Q,\nabla Q\ra = c_6(\dot b + b^2 - 2b(b+\lsl) - a)$, where $c_6\neq 0$, so that we obtain
\be\label{md5}
\Big|\dot b + b^2 - 2b(b+\lsl) - a \Big| \lesssim C^{**} s^{-2} \log^{-2}(s)+s^{-1} |\vec m^a_1|. 
\ee

\smallskip

\noindent Combining~\eqref{md1}--\eqref{md5}, we have proved, for all  $s\in[s^{**},s^{in}]$,
\begin{equation}\label{dstar}
|\vec m^a_1(s)|\lesssim C^{**} s^{-2} \log^{-2}(s).
\end{equation}

\noindent\textbf{step 4} Minimal mass property of the approximate solution. The proof of the degeneracy estimate~\eqref{b:os} relies on the following minimal mass property for the ansatz $\PP$ under the bootstrap assumptions~\eqref{BS}:
\be\label{b:mP}
\left| \|\PP(s)\|_{L^2} - \|\PP^{in}\|_{L^2} \right| \lesssim s^{-2}\log^{-2}(s).
\ee
Note that the implicit constant on the right-hand side does not depend on $C^{**}$. By the definition~\eqref{d:psik} of $\EE_\PP$, we have
\[
\frac 12\frac{d}{ds}\|\PP\|_{L^2}^2 = \la i \dot \PP,i\PP\rangle = \la \EE_\PP,i\PP\ra.
\]
In view of the formula for $\EE_\PP$~\eqref{d:psik}, and the definition of $\PP= \sum_j P_j$,~\eqref{b:mP} follows by integration of the following estimate: for all $j,k\in \{1,\ldots,K\}$,
\be\label{jk}
|\la \big[e^{i\Gamma_k} \Psi_k\big] (y-z_k),i \big[e^{i\Gamma_j} Q_a\big](y-z_j)\ra|\lesssim s^{-3} \log^{-2}(s).
\ee

\noindent\emph{Proof of~\eqref{jk}}. We start by proving~\eqref{jk} in the case $j=k=1$.
From~\eqref{d:psik}:
\[
\la \big[e^{i\Gamma_k} \Psi_k \big](y-z_k),i \big[e^{i\Gamma_k} Q_a\big](y-z_k)\ra 
= \la \Psi_1,iQ_a\ra = \la \vec m^a_1\cdot \vec\md Q_a+i \dot z a'(z) \rho+G_1+\Psi_{Q_a},iQ_a\ra.
\]
Note that $\la\vec\md Q,iQ\ra=0$. Thus, by~\eqref{dstar},~\eqref{BS}--\eqref{BS2},
\[
|\la \vec m^a_1\cdot \vec\md Q_a ,iQ_a\ra|\lesssim |a| |\vec m^a_1|\lesssim C^{**} s^{-4} \log^{-3}(s)
\lesssim s^{-3}\log^{-3}(s).
\]
Next, we claim the following estimate, which justifies   the special choice of $a(z)$ done in~\eqref{d:a}
(see also Sect.~\ref{s:2.2})
\be\label{c:a}
| \la i \dot z a'(z) \rho+G_1  ,iQ_a\ra|\lesssim s^{-3} \log^{-2}(s).
\ee
Indeed, first by~\eqref{BS}--\eqref{BS2} and~\eqref{b:iGk},
\be\label{tr}
|\la G_1,iQ_a\ra + \kappa c_a \la \rho,Q\ra b z^{\frac 32} e^{-\kappa z} | \lesssim
s^{-3} \log^{-2}(s).
\ee
Second, we note that by~\eqref{dstar} and~\eqref{BS}--\eqref{BS2},
\be\label{sm}
|\dot z - bz|\lesssim s^{-1} \log^{-1}(s),
\ee
 and that by the definition of $a(z)$ in~\eqref{d:a},
\be\label{ma}
|a'(z)-c_a \kappa z^{\frac 12} e^{-\kappa z} |\lesssim  z^{-\frac 12} e^{-\kappa z}\lesssim s^{-2}\log^{-2}(s).
\ee
Gathering~\eqref{tr}--\eqref{ma}, we obtain~\eqref{c:a}.
Finally, since $Q_a$ and $\Psi_{Q_a}$ given by~\eqref{d:PQa} are real-valued, we have the cancellation
\[
\la\Psi_{Q_a},iQ_a\ra=0.
\]
The collection of above estimates concludes the proof of~\eqref{jk} for $j=k=1$.

\noindent
We now prove~\eqref{jk} in the case $k=1$ and $j\in \{2,\ldots,K\}$.
Note that
\begin{align*}
&|\la \big[e^{i\Gamma_k} \Psi_k\big] (y-z_k),i \big[e^{i\Gamma_j} Q_a\big](y-z_j)\ra|\\
&= |\la \vec m^a_1\cdot \vec\md Q_a+i \dot z a'(z) \rho+G_1+\Psi_{Q_a} ,i e^{i(\Gamma_j(y-(z_j-z_k))-\Gamma_k(y))} Q_a(y-(z_j-z_k))\ra|
\end{align*}
First, by~\eqref{dstar}, for some $q>0$,
\begin{align*}
&|\la \vec m^a_1\cdot \vec\md Q_a,i e^{i(\Gamma_j(y-(z_j-z_k))-\Gamma_k)} Q_a(y-(z_j-z_k))\ra|
\\&\lesssim |\vec m^a_1| z^q e^{-\kappa z} \lesssim C^{**} s^{-4} \log^q (s) 
\lesssim s^{-3} \log^{-2}(s).
\end{align*}
Second, using similar arguments,  for some $q>0$,
\[
|\la  i \dot z a'(z) \rho+G_1+\Psi_{Q_a} ,i e^{i(\Gamma_j(y-(z_j-z_k))-\Gamma_k)} Q_a(y-(z_j-z_k))\ra|
\lesssim s^{-4}\log^q (s).
\]
The collection of above estimates concludes the proof of~\eqref{b:mP}.

\smallskip

\noindent{\bf step 5} Proof of~\eqref{b:os}. The conservation of mass for the solution $u$ and~\eqref{init} imply:
\[
\|u(s)\|_{L^2}   =\|u(s^{in})\|_{L^2} = \|\PP^{in}\|_{L^2}.
\]
By ~\eqref{dec},
\[
\la\e(s),\PP\ra = \frac{1}{2}\left( \|u(s)\|_{L^2}^2 - \|\PP(s)\|_{L^2}^2-\|\e(s)\|_{L^2}^2\right).
\]
Therefore, using~\eqref{BS}--\eqref{BS2} and~\eqref{b:mP}, we obtain
\[
|\la\e(s),\PP\ra|\lesssim s^{-2}\log^{-2}(s).
\]
Now, we use the symmetry $\la \e,P_k\ra = \la \e,P_j\ra=K^{-1} \la\e(s),\PP\ra$ for all $j,k\in \{1,\ldots,K\}$.
Moreover, by~\eqref{BS}--\eqref{BS2},
\[
\la\e(s),P_1\ra = \la \eta_1,Q_a\ra = \la \eta_1,Q\ra + O(|a|\|\e\|_{L^2})
=\la \eta_1,Q\ra + O(s^{-3}\log^{-\frac 52}(s)).
\]
Gathering this information, we obtain $|\la \eta_1,Q\ra |\lesssim s^{-2}\log^{-2}(s)$,
i.e. estimate~\eqref{b:os}. In particular, choosing $C^{**}$ large enough, we have $s^{**}=s^*$.

\smallskip

\noindent{\bf step 6} Conclusion. The estimate~\eqref{b:ode1} is   a direct consequence of~\eqref{b:m} and~\eqref{BS}--\eqref{BS2}. We now turn to the  proof of~\eqref{b:EE}. Using~\eqref{b:m},~\eqref{BS}--\eqref{BS2} and~\eqref{d:rh},
\[
|\dot z a'(z) \rho| 
\lesssim Q^{\frac 12} (|b| z + s^{-1} \log^{-1}(s)) s^{-2} \log^{-1}(s) 
\lesssim Q^{\frac 12} s^{-3} \log^{-1}(s).
\]
By~\eqref{b:m},
\[
|\vec m^a_1 \cdot \vec {\rm M} Q_a|\lesssim Q^{\frac 12} s^{-2}\log^{-2}(s).
\]
Next, by the definition of $G_k$ in~\eqref{d:Gk}, the decay $|\rho|\lesssim Q^{\frac 78}$ (see~\eqref{d:rh}) and $|{\rm e}_k-{\rm e}_1| \geq \kappa$ for $k\neq 1$,
\begin{align*}
|G_1|&\lesssim Q^{\frac 12}   \sum_{k=2}^K
\left(Q^{\frac 54}(y)Q(y-z(\ek-{\rm e}_1)) + |a| Q^{\frac 54}(y)Q^{\frac 78}(y-z(\ek-{\rm e}_1))\right)
\\& \lesssim Q^{\frac 12} (z^{-\frac 12} e^{-\kappa z} + s^{-2} e^{-\frac \kappa 2 z})
 \lesssim Q^{\frac 12} s^{-2} \log^{-2}(s).
\end{align*}
Finally, by the definition of $\Psi_{Q_a}$ in~\eqref{d:PQa},
\[
|\Psi_{Q_a}|\lesssim Q^{\frac 12} |a|^2 \lesssim Q^{\frac 12} s^{-4} \log^{-2}(s).
\]
The same estimates hold for $\nabla \EE$, which finishes the proof of~\eqref{b:EE}.
\end{proof}
 
 
\subsection{Energy functional}\label{s:5}


Consider the   nonlinear energy functional for $\e$ 
\[
\HH(s,\e) =
\frac 12 \int \left( |\nabla \e|^2 + |\e|^2 - \frac 12 \left( |\PP+\e|^4 - |\PP|^4 - 4 |\PP|^2 \re(\e \overline \PP)\right)\right).
\]
Pick a smooth function $\chi:[0,+\infty)\to [0,\infty)$, non increasing, with $\chi\equiv 1$ on $[0,\frac 1{10}]$, $\chi\equiv 0$ on $[\frac 18,+\infty)$. We define the localized momentum:
\[
\JJ = \sum_{k} J_k,\quad J_k (s,\e) =  b \im \int   ( z_k\cdot   \nabla \e) \bar \e \chi_k,
\quad \chi_k (s,y) = \chi\left(\log^{-1}(s) |y-z_k(s)|\right).
\]
Finally, set
\[
\FF(s,\e) = \HH(s,\e) - \JJ(s,\e).
\]
The functional $F$ is coercive in $\e$ at the main order and it is an almost conserved quantity for the problem.

\begin{proposition}[Coercivity and time control of the   energy functional]
\label{p:coe}
For all $s\in [s^{*},s^{in}]$,
\be
\label{coercvff}
\FF(s,\e(s)) \gtrsim \|\e(s)\|_{H^{1}}^{2} + O(s^{-4}\log^{-4}(s)),
\ee
and 
\be
\label{monotnoivnon}
\left|\frac{d}{ds} [\FF(s,\e(s))]\right| \lesssim s^{-2} \log^{-2}(s) \|\e(s)\|_{H^1} + s^{-1} \log^{-1}(s) \|\e(s)\|_{H^1}^2.
\ee
\end{proposition}

\begin{proof}[Proof of Proposition~\ref{p:coe}]
\noindent{\bf step 1} Coercivity. The proof of the coercivity~\eqref{coercvff} is   a standard consequence of the coercivity property~\eqref{e:co} around one solitary wave with the orthogonality properties~\eqref{ortho},~\eqref{b:os}, and an elementary localization argument. Hence we briefly sketch the argument. First, using the coercivity property~\eqref{e:co} and the orthogonality properties~\eqref{ortho},~\eqref{b:os} and localization arguments,
we have
\be\label{c:H}
\HH(s,\e) \gtrsim \|\e\|_{H^{1}}^{2} + O(s^{-4}\log^{-4}(s)).
\ee
Note that the error term $O(s^{-4}\log^{-4}(s))$ is due to the fact that the bound~\eqref{b:os} replaces 
a true orthogonality $\la \eta_1(s), Q\ra=0$. We refer to the proof of Lemma~4.1 in Appendix B of \cite{MMT2} for a similar proof. Second, we note that by~\eqref{BS}--\eqref{BS2}, $|\JJ(s,\e)|\lesssim |b| z \|\e\|_{H^1}^2 \lesssim s^{-1} \|\e\|_{H^1}^2 $, and~\eqref{coercvff} follows.

\smallskip

\noindent{\bf step 2} Variation of the energy. We   estimate the time variation of the functional  $\HH$  and claim: for all $s\in [s^{*},s^{in}]$,
\begin{multline}\label{estderh}
 \left| \frac{d}{ds}[\HH(s,\e(s))] - \sum_k \la \dot z_k \cdot \nabla P_k,  2 |\e|^2P_k + \e^2 \overline P_k\ra \right| 
\\
 \lesssim 
 s^{-2} \log^{-2}(s) \|\e(s)\|_{H^1} + s^{-1} \log^{-1}(s) \|\e(s)\|_{H^1}^2.
\end{multline}
The time derivative of $s\mapsto H(s,\e(s))$ splits into two parts
\[
\frac{d}{ds}[\HH(s,\e(s))]  = D_s \HH(s,\e(s)) + \la D_\e \HH(s,\e(s)),\dot \e_s\ra,
\]
where $D_s$ denotes differentiation of $\HH$ with respect to $s$ 
and $D_\e$ denotes differentiation of $\HH$ with respect to $\e$. First compute:
\[
D_s \HH = - \la \dot \PP,|\PP+\e|^2(\PP+\e)-|\PP^2|\PP - (2\e |\PP|^2 + \overline \e \PP^2) \ra.
\]
Observe that by the definition of $P_k$ in~\eqref{d:Pk},
\[
\dot P_k = - \dot z_k \cdot \nabla P_k + i (\dot \beta_k\cdot  (y-z_k) - \tfrac{\dot b}4 |y-z_k|^2  ) P_k + \dot z a'(z) \rho_k
\quad \hbox{where}\quad \rho_k=\big[e^{i\Gamma_k}\rho\big](y-z_k).\]
By~\eqref{b:ode1},~\eqref{BS}--\eqref{BS2} and~\eqref{d:a},
\[
|\dot \beta_k|+| {\dot b}| +| \dot z a'(z)|\lesssim s^{-2} \log^{-2}(s).
\]
Since
\[
\int \left| |\PP+\e|^2(\PP+\e)-|\PP^2|\PP - (2\e |\PP|^2 + \overline \e \PP^2)\right| \lesssim \|\e\|_{H^1}^2,
\]
we obtain
\begin{align*}
&\left| \left\la  i(\dot \beta_k\cdot  (y-z_k) - \tfrac{\dot b}4 |y-z_k|^2  ) P_k + \dot z a'(z) \rho_k,
|\PP+\e|^2(\PP+\e)-|\PP^2|\PP - (2\e |\PP|^2 + \overline \e \PP^2)\right\ra\right| \\ &\lesssim s^{-2} \log^{-2}(s) \|\e\|_{H^1}^2.
\end{align*}
Next, note that 
\[
| \PP+\e|^2(\PP+\e)-|\PP^2|\PP - (2\e |\PP|^2 + \overline \e \PP^2) = 2 |\e|^2 \PP +   \e^2 \overline \PP + |\e|^2\e .
\]
By ~\eqref{b:ode1} and~\eqref{BS}--\eqref{BS2}, $|\dot z|\lesssim s^{-1}$ and thus by~\eqref{BS}--\eqref{BS2},
\[
\left|  \la \dot z_k \cdot \nabla P_k ,  |\e|^3 \ra   \right| 
\lesssim s^{-1}\|\e\|_{H^1}^3 \lesssim s^{-2} \log^{-\frac 32}(s) \|\e\|_{H^1}^2.
\]
For $j\neq k$, since $e^{-\kappa z} \lesssim s^{-2}$ by~\eqref{BS}--\eqref{BS2} and the decay properties of
$P_k$, $P_j$,
\[
\left|\la \dot z_k \cdot \nabla P_k,2 |\e|^2 P_j +   \e^2 \overline P_j \ra \right| \lesssim 
|s|^{-3} \|\e\|_{H^1}^2.
\]
Gathering these computations, we have obtained
\begin{equation}\label{v1}
D_s \HH(s,\e) = 
\sum_k \la \dot z_k \cdot \nabla P_k,  2 |\e|^2P_k + \e^2 \overline P_k\ra + O(s^{-2} \|\e\|_{H^1}^2).
\end{equation}
Second,
\[
D_\e \HH(s,\e) = - \Delta \e + \e - \left( |\PP+\e|^2(\PP+\e) - |\PP|^2 \PP\right),
\]
so that the equation~\eqref{e:e} of $\e$ rewrites
\[
i\dot \e  -D_\e \HH(s,\e) -i\lsl \Lambda \e + (1-\dot \gamma)\e +\EE_{\PP}=0.
\]
In particular,
\begin{align*}
 \la D_\e \HH(s,\e),\dot \e\rangle
 &=  \la i D_\e \HH(s,\e),i\dot \e\ra
 \\ & =  \lsl \la D_\e \HH(s,\e), \Lambda \e \ra 
 - (1-\dot\gamma) \la i D_\e \HH(s,\e), \e \ra - \la i D_\e \HH(s,\e),\EE_{\PP}\ra.
\end{align*}
We recall that 
\[
\la -\Delta \e,\Lambda \e \ra = \|\nabla \e\|^2,\quad \la \e,\Lambda \e \ra = 0,\quad 
\la |\e|^2 \e,\Lambda \e\ra = \frac 12 \int |\e|^4,
\]
and thus, using also~\eqref{BS}--\eqref{BS2},~\eqref{GN}, and the decay properties of $Q$,
\[
\left| \la D_\e \HH(s,\e), \Lambda \e \ra\right| \lesssim \|\e\|_{H^1}^2 + \|\e\|_{H^1}^4\lesssim \|\e\|_{H^1}^2.
\]
In particular, from~\eqref{b:m} and~\eqref{BS}--\eqref{BS2}, we deduce
\[
\left|\lsl \la D_\e \HH(s,\e), \Lambda \e \ra  \right|  \lesssim s^{-1} \log^{-1}(s) \|\e\|_{H^1}^2.
\]
Note that the estimate on $b$ in~\eqref{BS}--\eqref{BS2} implies $|b|\lesssim s^{-1}\log^{-1}(s)\ll s^{-1}$ which avoids the use of virial localized identities (as in \cite{RaSz11, KLR}) to control   the above term.
By~\eqref{b:m} and~\eqref{BS}--\eqref{BS2}, we estimate
\[
| (1-\dot\gamma) \la i D_\e \HH(s,\e), \e \ra | \lesssim s^{-2}  \|\e\|_{H^1}^2.
\]
Finally, integrating by parts, using~\eqref{b:EE} and~\eqref{BS}--\eqref{BS2}, we have
\[
| \la i D_\e \HH(s,\e),\EE_{\PP}\ra|
\lesssim \la |\nabla \e|,|\nabla \EE_{\PP}|\ra+\la |\e|+|\e|^3,|\EE_{\PP}|\ra
\lesssim
s^{-2}\log^{-2}(s) \|\e\|_{H^1}.
\]
The collection of above estimates finishes the proof of~\eqref{estderh}.

\smallskip

\noindent{\bf step 3} Variation of the localized momentum. We now claim: for all $s\in [s^{*},s^{in}]$,
\begin{multline}\label{estmomentum}
\left| \frac{d}{ds}[\JJ(s,\e(s))] - b \sum_k \la   z_k \cdot \nabla P_k,  2 |\e|^2P_k + \e^2 \overline P_k\ra \right| 
\\
\lesssim  
 s^{-2} \log^{-2}(s) \|\e(s)\|_{H^1} + s^{-1} \log^{-1}(s) \|\e(s)\|_{H^1}^2.
\end{multline}
Indeed, we compute, for any $k$,
\begin{align*}
\frac{d}{ds}[J_k(s,\e(s))] 
& = \dot b  \im \int   ( z_k\cdot   \nabla \e) \bar \e \chi_k + b  \im \int   ( \dot z_k\cdot   \nabla \e) \bar \e \chi_k
+ b \im \int   ( z_k\cdot   \nabla \e) \bar \e \dot \chi_k \\
& + b \la i\dot \e, z_k\cdot (2 \chi_k \nabla \e + \e \nabla \chi_k)\ra.
\end{align*}
By~\eqref{b:m} and~\eqref{BS}--\eqref{BS2}, we have
\[
\left|  \dot b  \im \int   ( z_k\cdot   \nabla \e) \bar \e \chi_k \right|
+\left| b  \im \int   ( \dot z_k\cdot   \nabla \e) \bar \e \chi_k \right|
\lesssim s^{-2} \|\e\|_{H^1}^2.
\]
Note that by direct computations,~\eqref{b:m} and~\eqref{BS}--\eqref{BS2},
\begin{align*}
|\dot \chi_k | &\lesssim (s^{-1}\log^{-1}(s)|y-z_k| + |\dot z_k|) \log^{-1}(s)|\chi'(\log^{-1}(s)(y-z_k(s)))|
 \lesssim s^{-1} \log^{-1}(s)
\end{align*}
and so, by~\eqref{BS}--\eqref{BS2},
\[
\left|b \im \int   ( z_k\cdot   \nabla \e) \bar \e \dot \chi_k  \right| \lesssim 
s^{-2}\log^{-2}(s) \|\e\|_{H^{1}}^{2}.
\]
Now, we use the equation~\eqref{e:e} of $\e$ to estimate $b \la i\dot \e, z_k\cdot (2 \chi_k \nabla \e + \e \nabla \chi_k)\ra$.
By integration by parts, we check the following
\[
\la \Delta \e, 2 (z_k\cdot \nabla \e) \chi_k\ra = 
\int |\nabla \e|^{2} (z_k\cdot \nabla \chi_k) - 2 \la (\nabla \e \cdot \nabla \chi_k),(z_k\cdot \nabla \e)\ra,
\]
\[
\la \Delta \e,\e(z_k\cdot \nabla \chi_k)\rangle= 
-\int |\nabla \e|^{2} (z_k\cdot \nabla \chi_k) +\frac 12 \int |\e|^{2} (z_k\cdot \nabla (\Delta \chi_k)).
\]
Thus,
\[
\la \Delta \e, z_k\cdot (2 \chi_k \nabla \e + \e \nabla \chi_k)\ra = 
- 2 \la (\nabla \e \cdot \nabla \chi_k),(z_k\cdot \nabla \e)\rangle
+\frac 12 \int |\e|^{2} (z_k\cdot \nabla (\Delta \chi_k)).
\]
By~\eqref{BS}--\eqref{BS2}, $|b|\lesssim  s^{-1} \log^{-1}(s)$ and $|z_k|\lesssim \log(s)$. Moreover, $|\nabla \chi_k|\lesssim \log^{-1}(s)$. Therefore,
\[
\left| b \la (\nabla \e \cdot \nabla \chi_k),(z_k\cdot \nabla \e)\rangle\right|
\lesssim s^{-1}\log^{-1}(s)\|\e\|_{H^1}^{2}.
\]
Similarly, by $|\nabla (\Delta \chi_k)|\lesssim \log^{-3}(s)$, we obtain
\[
\left| b \int |\e|^{2} (z_k\cdot \nabla (\Delta \chi_k))\right|
\lesssim s^{-1}\log^{-3}(s)\|\e\|_{H^1}^{2}.
\]
In conclusion for term $\Delta \e$ in the equation of $\e$, we obtain
\[
|b \la \Delta \e, z_k\cdot (2 \chi_k \nabla \e + \e \nabla \chi_k)\ra| 
\lesssim s^{-1}\log^{-1}(s)\|\e\|_{H^1}^{2}.
\]
For the mass term in the equation of $\e$, we simply check by integration by parts that
\[
\la   \e, z_k\cdot (2 \chi_k \nabla \e + \e \nabla \chi_k)\ra =0.
\]
We also check that
\[
\la  i\Lambda \e, z_k\cdot (2 \chi_k \nabla \e + \e \nabla \chi_k)\ra =
2 \la i \e,(z_k\cdot \nabla \e)\chi_k\ra + \la i (y\cdot \nabla \e), \e (z_k\cdot \nabla \chi_k)\ra,
\]
and thus, by~\eqref{BS}--\eqref{BS2},
\begin{align*}
|b| \left| \lsl\right| \left|\la  i\Lambda \e, z_k\cdot (2 \chi_k \nabla \e + \e \nabla \chi_k)\ra \right| \lesssim 
s^{-2} \log^{-1}(s)  \|\e\|_{H^1}^2 .
\end{align*}
Next, from~\eqref{b:EE}, 
\[
| b   \la\EE_\PP, z_k\cdot (2 \chi_k \nabla \e + \e \nabla \chi_k) \ra |\lesssim s^{-3} \log^{-2}(s) \|\e\|_{H^1}.
\]
Now, we only have to deal with the term
\[
b \la   |\PP+\e|^2(\PP+\e) - |\PP|^2 \PP ,  z_k\cdot (2 \chi_k \nabla \e + \e \nabla \chi_k)  \ra .
\]
Recall that
$
| \PP+\e|^2(\PP+\e)-|\PP^2|\PP = (2\e |\PP|^2 + \overline \e \PP^2) + 2 |\e|^2 \PP +   \e^2 \overline \PP + |\e|^2 \e.
$
First, by~\eqref{BS}--\eqref{BS2}, it is clear that
\[
\left| b \la  2 |\e|^2 \PP +   \e^2 \overline \PP + |\e|^2\e   ,  z_k\cdot (2 \chi_k \nabla \e + \e \nabla \chi_k)  \ra \right|
\lesssim s^{-1} \|\e\|_{H^1}^3 \lesssim s^{-2}\log^{-\frac 32}(s) \|\e\|_{H^1}^2.
\]
Second, since $|b|\lesssim  s^{-1} \log^{-1}(s)$, $|z_k|\lesssim \log(s)$ and  $|\nabla \chi_k|\lesssim \log^{-1}(s)$, 
\[
\left| b \la 2\e |\PP|^2 + \overline \e \PP^2,\e (z_k\cdot \nabla \chi_k)  \ra \right|
\lesssim s^{-1} \log^{-1}(s) \|\e\|_{H^1}^2.
\]
Third, by the decay property of $Q$ and the definition of $\chi_k$,
\[
\Big|b\Big\la  2\e \Big(|\PP|^2-\sum_j |P_j|^2 \Big)  + \overline \e  \Big(\PP^2- \sum_j P_j^2 \Big),
 (z_k\cdot  \nabla \e) \chi_k\Big\ra \Big|
\lesssim s^{-2} \|\e\|_{H^1}^2,
\]
and, for $j\neq k$,
\[
\Big|b\Big\la  2\e |P_j|^2   + \overline \e   P_j^2 ,
 (z_k\cdot  \nabla \e) \chi_k\Big\ra \Big|
\lesssim s^{-2} \|\e\|_{H^1}^2.
\]
Finally, we compute by integration by parts,
\begin{align*}
 \la 2 \e |P_k|^2 + \overline \e P_k^2 , (z_k\cdot \nabla \e) \chi_k\ra
& =  - \la  z_k \cdot \nabla P_k,  2 |\e|^2P_k + \e^2 \overline P_k\ra\\
& -\frac 12 \re\left( \int \left(2|\e|^2 P_k^2 + \e^2 |P_k|^2 \right) (z_k\cdot \nabla \chi_k)\right).
\end{align*}
As before,
\[
\left| b
\re\left( \int \left(2|\e|^2 P_k^2 + \e^2 |P_k|^2 \right) (z_k\cdot \nabla \chi_k)\right)
\right| \lesssim s^{-1} \log^{-1}(s) \|\e\|_{H^1}^2.
\]
The collection of above bounds concludes the proof of~\eqref{estmomentum}.

\smallskip

\noindent{\bf step 3} Conclusion. Recall that, by~\eqref{b:ode1},
$
|\dot z_k - b z_k|\lesssim s^{-1}\log^{-1}(s)$, and so
\[
\left| \la   (\dot z_k - b z_k ) \cdot \nabla P_k,  2 |\e|^2P_k + \e^2 \overline P_k\ra \right|
\lesssim s^{-1}\log^{-1}(s) \|\e\|_{H^1}^2,
\]
and~\eqref{monotnoivnon} now follows from~\eqref{estderh},~\eqref{estmomentum}. This concludes the proof of Proposition~\ref{p:coe}.
\end{proof}


\subsection{End of the proof of Proposition~\ref{p:st}}\label{:s:4.4}

  We   close the bootstrap estimates~\eqref{BS} and   prove~\eqref{b:su}.

\smallskip

\noindent{\bf step 1} Closing the estimates in $\e$. By~\eqref{monotnoivnon} in Proposition~\ref{p:coe} and then~\eqref{BS}--\eqref{BS2}, we have
\[
\left|\frac{d}{ds}[ \FF(s,\e(s))]\right| \lesssim s^{-2}\log^{-2}(s) \|\e\|_{H^1} +s^{-1} \log^{-1}(s) \|\e\|_{H^1}^2 \lesssim s^{-3} \log^{-\frac 72}(s).
\]
Thus, by integration on $[s,s^{in}]$ for any $s\in [s^*,s^{in}]$, using $\e(s^{in})=0$ (see~\eqref{c:in2}), we obtain
\[
|\FF(s,\e(s))| \lesssim  s^{-2} \log^{-\frac 72}(s).
\]
By~\eqref{coercvff} in Proposition~\ref{p:coe}, we obtain
\[
\|\e(s)\|_{H^1}^2 \lesssim  s^{-2} \log^{-\frac 72}(s) .
\]
Therefore, for $s_0$ large enough, for all $s\in [s^*,s^{in}]$,
\[
\|\e(s)\|_{H^1}^2 \leq \frac 12 s^{-2} \log^{-3} (s) ,
\]
which strictly improves the estimate on $\|\e\|_{H^1}^2$ in~\eqref{BS}.

\smallskip

\noindent{\bf step 2} Closing the parameter estimates. First, note that from~\eqref{b:ode1}, 
$
|\dot \beta|\lesssim s^{-2} \log^{-2} (s).
$
Together with the choice $\beta(s^{in})= \beta^{in}= 0$ (see~\eqref{c:in}), direct integration in time gives, for all $s\in [s^*,s^{in}]$,
$
|\beta(s)|\lesssim s^{-1} \log^{-2}(s).
$
For $s_0$ large enough, we obtain, for all $s\in [s^*,s^{in}]$,
\[
|\beta(s)|< \frac 12 s^{-1} \log^{-\frac 32}(s), \]
which strictly improves the estimate on $\beta(s)$ in~\eqref{BS}.

\smallskip

Second, recall from~\eqref{b:ode1},~\eqref{BS2} and the definition of $a(z)$ in~\eqref{d:a},
for all $s\in [s^*,s^{in}]$,
\[
\left| \dot b + c_a z^{\frac 12} e^{-\kappa z} \right| \lesssim s^{-2} \log^{-2}(s),\quad
\left| \dot z z^{-1} - b \right| \lesssim s^{-1} \log^{-2}(s).
\]
Since $|\dot b|\lesssim s^{-2} \log^{-1}(s)$ and $| \dot z z^{-1}|\lesssim s^{-1}\log^{-1}(s)$, it follows that
\[
\left| \dot b b + c_a \dot z z^{-\frac 12} e^{-\kappa z} \right| \lesssim s^{-3} \log^{-3}(s).
\]
Integrating on $[s,s^{in}]$ for any $s\in [s^*,s^{in})$, using the special relation between $b^{in}$ and $z^{in}$ fixed in~\eqref{c:in}
\[
b^2(s^{in}) = \frac {2 c_a}\kappa z^{-\frac 12}(s^{in}) e^{-\kappa z(s^{in})},\quad b(s^{in})>0,
\]
we obtain
\be\label{sam}
\left| b^2  - \frac {2 c_a}\kappa z^{-\frac 12}  e^{-\kappa z}\right|
\lesssim s^{-2} \log^{-3}(s) + \int_s^{s^{in}} \left| \dot z z^{-\frac 32} e^{-\kappa z}\right| ds'
\lesssim s^{-2} \log^{-3}(s),\quad b(s)>0.
\ee
From~\eqref{BS}--\eqref{BS2} and~\eqref{BS2}, we have
\[
\left|  \frac{2c_a}{\kappa} z^{-\frac 12} e^{- {\kappa} z} -  s^{-2}\log^{-2}(s)\right| 
\lesssim s^{-2}\log^{-\frac 52} s.
\]
Therefore, the following estimate on $b(s)$  follows from~\eqref{sam} 
\[
\left| b^2  -   s^{-2}\log^{-2}(s)  \right|
\lesssim s^{-2} \log^{-\frac 52}(s) . 
\]
This implies, for all $s\in [s^*,s^{in}]$,
\be\label{sam3}
\left| b  -   s^{-1}\log^{-1}(s)  \right|
\lesssim s^{-1} \log^{-\frac 32}(s) ,
\ee
which   strictly improves the estimate on $b(s)$ in~\eqref{BS}.

\smallskip

Finally, we address the estimate on $z(s)$.
From~\eqref{sam},~\eqref{BS}--\eqref{BS2} and~\eqref{b:ode1}, we have
\be\label{sam2}
\left| b  - \left(\frac{2 c_a}{\kappa}\right)^{\frac 12} z^{-\frac 14} e^{-\frac \kappa 2 z}\right|
+ \left| \dot z z^{-1}  - \left(\frac{2 c_a}{\kappa}\right)^{\frac 12} z^{-\frac 14} e^{-\frac \kappa 2 z}\right|
\lesssim s^{-1} \log^{-2}(s).
\ee
Using $z\lesssim \log^{-1}(s)$, we obtain
\be\label{dd}
\left| \frac{d}{ds}\left( z^{-\frac 34} e^{\frac \kappa 2 z}\right)  - \left(\frac{\kappa c_a}{2}\right)^{\frac 12} \right|
\lesssim \log^{-1}(s)  + \left|\dot z z^{-\frac 74} e^{\frac \kappa 2 z}\right|
\lesssim \log^{-1}(s).
\ee
We need to adjust the initial choice of $z(s^{in})=z^{in}$ through a topological argument (see \cite{CMM} for a similar argument).
We define $\zeta $ and $\xi$ the following two functions on $[s^*,s^{in}]$
\be\label{d:zx}
\zeta(s) = \left(\frac2{\kappa c_a}\right)^{\frac 12} z^{-\frac 34} e^{\frac \kappa 2 z},\quad
\xi(s) = (\zeta(s)-s)^2 s^{-2} \log(s).
\ee
Then,~\eqref{dd} writes
\be\label{ddd}
| \dot \zeta(s) - 1|\lesssim \log^{-1}(s).
\ee
According to~\eqref{BS}, our objective is to prove that there exists a suitable choice of \[\zeta(s^{in})=\zeta^{in}\in [s^{in}-s^{in}\log^{-\frac{1}{2}}(s^{in}),s^{in}+s^{in}\log^{-\frac{1}{2}}(s^{in})],\]
so that $s^*=s_0$.
Assume for the sake of contradiction that for all $\zeta^{\sharp}\in [-1,1]$, the choice
\[
\zeta^{in} = s^{in} + \zeta^{\sharp} s^{in} \log^{-\frac{1}{2}}(s^{in})
\]
 leads to $s^*=s^*(\zeta^{\sharp})\in (s_0,s^{in})$.
Since all estimates in~\eqref{BS} except the one on $z(s)$ have been strictly improved on $[s^*,s^{in}]$, it follows from  $s^*(\zeta^{\sharp})\in (s_0,s^{in})$ and continuity that
\[
|\zeta(s^*(\zeta^{\sharp})) - s^*| = s^{*} \log^{-\frac 12} s^*\quad \hbox{i.e.} \quad
\zeta(s^*(\zeta^{\sharp})) = s^* \pm s^{*} \log^{-\frac 12} s^*.
\]
We need a transversality condition to reach a contradiction.
We compute:
\[
\dot \xi(s) = 2 (\zeta(s)-s)(\dot \zeta(s) - 1) s^{-2}\log(s)
- (\zeta(s)-s)^{2} (2 s^{-3} \log(s) -s^{-3}).
\]
At $s=s^{*}$, this gives
\[
| \dot \xi(s^{*}) + 2 (s^*)^{-1} |\lesssim (s^*)^{-1} \log^{-\frac 12}(s^*).
\]
Thus, for $s_0$ large enough,
\be\label{ng}
\dot \xi(s^{*}) < - (s^*)^{-1}.
\ee
Define the function $\Phi$ by 
\[
\Phi \, : \, \zeta^{\sharp}\in [-1,1] \mapsto 
(\zeta(s^{*})-s^{*})(s^{*})^{-1}\log^{\frac 12}(s^{*})\in \{-1,1\}.
\]
A standard consequence of the transversality property~\eqref{ng} is the continuity of the function $\zeta^{\sharp}\in[-1,1]\mapsto s^*(\zeta^{\sharp})$. In particular,   the function $\Phi$ is also continuous from $[-1,1]$ to $\{1,-1\}$. Moreover, for $\zeta^{\sharp}=-1$ and $\zeta^{\sharp}=1$, $\xi(s^{*})=1$ and $\dot \xi(s^{*})<0$ by~\eqref{ng} and so in these cases $s^{*} =s^{in}$. Thus, $\Phi(-1)=-1$ and $\Phi(1)=1$, but this is in contradiction with the continuity.

\smallskip

In conclusion, there exists at least a choice of 
 \[\zeta(s^{in})=\zeta^{in}\in (s^{in}-s^{in}\log^{-\frac{1}{2}}(s^{in}),s^{in}+s^{in}\log^{-\frac{1}{2}}(s^{in}))\]
such that $s^*=s_0$.

\smallskip

\noindent{\bf step 3} Conclusion. 
To finish proving~\eqref{b:su}, we only have to prove the estimate on $\lambda(s)$.
From~\eqref{b:m} and~\eqref{sam3}, we obtain
\[
\left|\lsl +    s^{-1}\log^{-1}(s) \right| \lesssim s^{-1} \log^{-\frac 32}(s).
\]
By integration on $[s,s^{in}]$, for any $s\in [s_0,s^{in})$, using the value $\lambda(s^{in})=\lambda^{in}=\log^{-1}(s^{in})$ (see~\eqref{c:in}),
we have
\[
\left| \log(\lambda(s)) +   \log(\log(s)) \right| \lesssim \log^{-\frac 12} (s),
\]
and thus
\be\label{sam4}
\left| \lambda(s) - \log^{-1}(s)\right| \lesssim \log^{-\frac 32} (s).
\ee
This concludes the proof of Proposition~\ref{p:st}.


\section{Compactness arguments}\label{s:4}


The objective of this section is to finish the construction of  Theorem~\ref{thmmain}
by passing to the limit on a sequence of solutions given by  Proposition~\ref{p:st}. 

\subsection{Construction of a sequence of backwards solutions}
We claim the following consequence of  Proposition~\ref{p:st}.

\begin{lemma}\label{le:csq}
There exist $t_0>1$  and
a sequence of solutions $u_n\in \mathcal C([t_0-T_n,0],\Sigma)$ of~\eqref{nls},
 where
\begin{equation}\label{Tn}
T_n \to +\infty \quad \hbox{as $n\to +\infty$},
\end{equation}
satisfying the following estimates, for all $t\in [t_0-T_n,0]$,
\be\label{b:tu}
\begin{aligned} 
&\left|z_n(t)- \frac{2}{\kappa}\log(t+T_n)\right|\lesssim \log(\log(t+T_n)),\quad
 \left| \lambda_n(t) - \log^{-1}(t+T_n)\right| \lesssim \log^{-\frac 32} (t+T_n),\\
& \left|b_n(t) -   (t+T_n)^{-1} \log^{-3}(t+T_n)\right| +|\beta_n(t)|+\|\e_n(t)\|_{H^1}\lesssim (t+T_n)^{-1} \log^{-\frac 72}(t+T_n), \\
& \left|a_n(t)\right|\lesssim (t+T_n)^{-2} \log^{-1}(t+T_n),\quad \e_n(0)\equiv 0,
\end{aligned}
\ee
where $(\lambda_n,z_n,\gamma_n,\beta_n,b_n)$ are  the parameters  of the decomposition of $u_n$ given by Lemma~\ref{le:mod}, i.e.
\begin{equation}\label{modn}
u_n(t,x) = \frac{e^{i\gamma_n(t)}}{\lambda_n(t)}\left(
\sum_k \big[e^{i\Gamma_{k,n}} Q_{a_n} \big]\left( \frac{x}{\lambda_n(t)}-z_n(t)\ek\right) 
+ \e_n\left(t, \frac{x}{\lambda_n(t)} \right)\right),
\end{equation}
with $\Gamma_{k,n}(t,y) = \beta_n(t) (\ek\cdot y) - \frac{b_n(t)}4 |y|^2$ and $Q_{a_n}=Q+a_n \rho$.
Moreover,  for all $t\in [t_0-T_n,0]$, 
\be\label{Sig}
\int |u_n(t,x)|^2 |x|^2 dx \lesssim 1.
\ee
\end{lemma}

\begin{proof}
Applying Proposition~\ref{p:st} with $s^{in}=n$ for any large $n$,  there exists a solution $u_n(t)$ of~\eqref{nls}
defined on the time interval $[-T_n,0]$ where
\[
T_n = \int_{s_0}^{n} \lambda_n^2(s) ds,
\]
and whose decomposition satisfies the uniform estimates~\eqref{b:su} on $[-T_n,0]$.
First, we see that~\eqref{Tn} follows directly from the estimate on $\lambda_n(s)$ in~\eqref{b:su}.

\smallskip

\noindent\emph{Proof of~\eqref{b:tu}.}
From~\eqref{b:su} and  the definition of the rescaled time $s$ (see~\eqref{d:s}),  for any $s\in [s_0,n]$,
we have
\begin{align*}
t(s)+T_n & = \int_{s_0}^{s} \lambda_n^2(s') ds' \quad\hbox{where}\quad
|\lambda_n^2(s) - \log^{-2}(s) |\lesssim  \log^{-\frac 52}(s).
\end{align*}
Fix $\bar s_0>s_0$   large enough independent of $n$ so that, for all $\bar s_0<s<n$,
\[
 \frac12   s \log^{-2}(s)\leq \int_{s_0}^{s} \lambda_n^2(s') ds'  =   s \log^{-2}(s) + O(s \log^{-\frac 52}(s))
 \leq \frac32   s \log^{-2}(s).
\]
Fix $t_0 =  \frac32   \bar s_0 \log^{-2}(\bar s_0)$. Then, for all $t\in [t_0-T_n,0]$,
\[
t+T_n = s\log^{-2}(s)\left(1+   O ( \log^{-\frac 12}(s) )\right)\geq \frac12   s \log^{-2}(s),
\]
and 
\[s = (t+T_n) \log^{2}(t+T_n) \left(1+O(\log^{-\frac 12}(t+T_n))\right).\]
Thus, estimates~\eqref{b:tu} are direct consequences of  ~\eqref{b:su}.

\smallskip

\noindent\emph{Proof of~\eqref{Sig}.}
 From~\eqref{modn} and   $\e_n(0)\equiv0$,  we have $u_n(0)\in \Sigma$.
It is then standard (see e.g.  \cite{Ca03}, Proposition~6.5.1) that $u_n\in \mathcal C([t_0-T_n,0],\Sigma)$.
We claim the following preliminary estimates. Fix $A=\frac {16} \kappa\geq 8 $.
For any $k\in \{1,\ldots,K\}$, for all $t\in [t_0-T_n,0]$,
\begin{equation}\label{ff1}
\frac 1{\lambda_n^2(t)} \int \left|Q_{a_n} \left(\frac x{\lambda_n(t)} - z_n(t) \ek \right)\right|^2 |x|^2 dx \lesssim 1,
\end{equation}
\begin{equation}\label{ff2}
\frac 1{\lambda_n^2(t)} \int_{|x|>A}\left|\nabla_x \left(\big[e^{i\Gamma_{k,n}} Q_{a_n}\big] \left(\frac x{\lambda_n(t)} - z_n(t) \ek \right)\right)\right|^2  dx \lesssim (t+T_n)^{-4}.
\end{equation}
Indeed,~\eqref{ff1} follows from a change of variable and the decay properties of $Q$ and $\rho$, 
\begin{multline*}
\frac 1{\lambda_n^2(t)} \int \left|Q_{a_n} \left(\frac x{\lambda_n(t)} - z_n(t) \ek \right)\right|^2 |x|^2 dx
= \int \left|Q_{a_n}(y)\right|^2 |\lambda_n(t) y + \lambda_n(t) z_n(t) \ek|^2 dy
  \lesssim 1.
\end{multline*}
where we have used from~\eqref{b:tu},
\begin{equation}\label{zl}
\lambda_n(t) z_n(t) \lesssim 1.
\end{equation}
To show~\eqref{ff2}, we see first that by~\eqref{b:tu},
\[
  \left| \nabla  \big[e^{i\Gamma_{k,n}} Q_{a_n}\big]  (y)  \right|^2 
 \lesssim   \left| \nabla  Q_{a_n} (y) \right|^2
 + \left(|\beta_n|^2  + b_n^2 |y|^2\right)   Q_{a_n}^2 (y) \lesssim e^{-\frac 32|y|}.
\]
Thus, by change of variable (using $A\geq 8$),
\begin{multline*}
\frac 1{\lambda_n^2(t)} \int_{|x|>A}\left|\nabla_x \left(\big[e^{i\Gamma_{k,n}} Q_{a_n}\big] \left(\frac x{\lambda_n(t)} - z_n(t) \ek \right)\right)\right|^2  dx \\ 
= \frac 1{\lambda_n^2(t)} \int_{|y+z_n(t) \ek|> A/\lambda_n(t)} \left| \nabla  \big[e^{i\Gamma_{k,n}} Q_{a_n}\big]  (y)  \right|^2 dy
\\\lesssim \log^2(t+T_n) \int_{|y|>\frac A2 \log(t+T_n)} e^{-\frac 32|y|} dy \lesssim (t+T_n)^{-\frac A2}
=(t+T_n)^{-4},
\end{multline*}
where we have used from~\eqref{b:tu} (possibly taking a larger $t_0$),
\[
|y+z_n(t) \ek|> \frac A{\lambda_n(t)} \quad \Rightarrow \quad
|y| > \frac A{\lambda_n(t)} - |z_n(t)| 
> \left( \frac {3A}4 - \frac 4\kappa\right) \log(t+T_n)\geq\frac A2 \log(t+T_n).
\]
Thus~\eqref{ff2} is proved. Observe that~\eqref{ff1}--\eqref{ff2} and~\eqref{b:tu} imply
\begin{equation}\label{ff3}
\|x u_n(0)\|_{L^2} \lesssim 1,\quad \|\nabla u_n(t)\|_{L^2(|x|>A)} \lesssim (t+T_n)^{-1} \log^{-\frac 52}(t+T_n).
\end{equation}

\smallskip

Define $\varphi:\RR^2 \to [0,1] $ by $\varphi(x) = (|x|-A)^2$ for $|x|>A$ and  $\varphi(x)=0$ otherwise.
By elementary computations, 
\[
\frac{d}{dt} \int |u_n|^2 \varphi  = 2 \im \int (\nabla \varphi  \cdot \nabla u_n) \overline u_n
 = 4 \int_{|x|>A} \left(\frac{x}{|x|}\cdot \nabla u_n\right) \overline u_n \varphi^{\frac 12} .
\]
Thus, by~\eqref{ff3},
\[
\left| \frac{d}{dt} \int |u_n|^2 \varphi\right| 
\lesssim  \left(\int |u_n|^2 \varphi \right)^{\frac 12} \left( \int_{|x|>A} |\nabla u_n(t)|^2\right)^{\frac 12}
\lesssim  (t+T_n)^{-1}\log^{-\frac 52} (t+T_n)\left(\int |u_n|^2 \varphi \right)^{\frac 12} .
\]
By integration and~\eqref{ff3}, the following uniform bound holds on $[t_0-T_n,0]$,
\[
\int |u_n(t,x)|^2 \varphi(x) dx\lesssim 1 \hbox{ and thus } \int |u_n(t,x)|^2 x^2 dx\lesssim 1,
\]
which finishes the proof of~\eqref{Sig}.
\end{proof}

\subsection{Compactness argument}  
By~\eqref{b:tu}--\eqref{Sig},   the sequence $(u_n(t_0-T_n))$ is bounded in $\Sigma$. Therefore, there exists a subsequence of $(u_n)$ (still denoted by $(u_n)$)
and $u_0\in \Sigma$ such that 
\[\begin{aligned}
&u_n(t_0-T_n) \rightharpoonup u_0 \quad \hbox{weakly in $H^1(\RR^2)$,}\\
&u_n(t_0-T_n) \rightarrow u_0 \quad \hbox{ in $H^\sigma(\RR^2)$ for $0\leq \sigma <1$, as $n\to +\infty$.}
\end{aligned}\]
Let $u $ be the solution of~\eqref{nls} corresponding to $u(t_0)=u_0$. By the local Cauchy theory for~\eqref{nls} (see \cite{Ca03} and \cite{CaWe90})
 and the properties of the sequence $u_n(t)$
(recall that $T_n\to \infty$), it follows that $u\in \mathcal C([t_0,+\infty),\Sigma)$.
Moreover, for all $0\leq \sigma <1$, for all $t\in [t_0,+\infty)$,
\[
u_n(t-T_n) \to u(t) \quad \hbox{in $H^\sigma$.}
\]
By weak convergence in $H^1$, $u(t)$ satisfies~\eqref{close} for all $t\geq t_0$. Moreover,  the decomposition $(\vec p,\e)$ of $u$ satisfies, for all $t\geq t_0$,
\be\label{co3}
\vec p_n(t-T_n) \to \vec p(t),\quad \e_n(t-T_n)\to \e(t) \hbox{\ in $H^\sigma$,}
\quad \e_n(t-T_n) \rightharpoonup \e(t) \hbox{\ in $H^1$}
\ee
(see e.g. \cite{MeRa04}, Claim p.598).
In particular, for all  $t\in [t_0,+\infty)$, $u(t)$ decomposes as
\begin{equation}\label{modu}
u(t,x) = \frac{e^{i\gamma(t)}}{\lambda(t)}\left(
\sum_k \big[e^{i\Gamma_{k}} Q_{a} \big]\left( \frac{x- \lambda(t)z(t)\ek }{\lambda(t)}\right) 
+ \e\left(t, \frac{x}{\lambda(t)} \right)\right),
\end{equation}
where $\Gamma_{k}(t,y) = \beta(t) (\ek\cdot y) - \frac{b(t)}4 |y|^2$ and
\be\label{b:tuO}
\begin{aligned} 
&\left|z(t)- \frac{2}{\kappa}\log(t)\right|\lesssim \log(\log(t)),\quad
 \left| \lambda(t) - \log^{-1}(t)\right| \lesssim \log^{-\frac 32} (t),\\
& \left|b(t) -   t^{-1} \log^{-3}(t)\right| + |\beta(t)|+\|\e(t)\|_{H^1} \lesssim t^{-1} \log^{-\frac 72}(t),
\quad |a(t)|\lesssim t^{-2} \log^{-1}(t), \\
& \int |u(t,x)|^2 |x|^2 dx \lesssim 1.
\end{aligned}
\ee 
Note that by~\eqref{b:tuO}, we have for all $k\in \{1,\ldots,K\}$,
\[
 x_k(t) = \lambda(t) z(t) \ek \to \frac 2\kappa \ek,\quad 
 \hbox{with}\quad
 \left| x_k(t)-\frac 2\kappa \ek\right|\lesssim \frac {\log(\log(t))}{\log(t)}.
\]
Since $\lambda^{-1}(t)\|\e(t)\|_{H^1}\lesssim t^{-1}\log^{-\frac 52}(t)$
and, by~\eqref{modu} and~\eqref{b:tuO}, 
\begin{equation}\label{ff6}
\lambda^{-1}(t) \left\| e^{i\Gamma_{k}} Q_{a} 
-Q\right\|_{H^1}\lesssim \lambda^{-1}(t) (|\beta(t)|+|b(t)|+|a(t)|)\lesssim t^{-1} \log^{-2}(t),
\end{equation}
 we obtain the following  stronger form of ~\eqref{strognconv}
\be
\label{conv2}
\left\|u(t)- e^{i\gamma(t)}\sum_k\frac{1}{\lam(t)}Q\left(\frac{.-x_k(t)}{\lam(t)}\right)\right\|_{H^1}\lesssim t^{-1}
\log^{-2}(t).
\ee 
Next, since for $j\neq k$, for some $q$,
\[
\lambda^{-2}(t)\int \left| \nabla Q\left(y- z(t) \ek\right)\cdot\nabla Q\left(y- z(t){\rm e}_j\right)\right|dy
\lesssim |z|^q e^{-  \kappa z} \lesssim t^{-1},
\]
we also obtain~\eqref{thm2}.
As a final remark, note that by global existence and  uniform bound in $\Sigma$,  the virial identity~\eqref{globalvirial} implies  the rigidity $E(u)=0$.
This concludes the proof of Theorem~\ref{thmmain}.

\subsection{Proof of Corollary~\ref{maincor}}
For $-t_0^{-1}<t<0$, we set
\begin{align*}
&\tilde z(t) = z(|t|^{-1}),\quad 
\tilde \lambda(t) = |t| \lambda(|t|^{-1}),\quad 
\tilde a(t) = a(|t|^{-1}),\quad 
\tilde b(t) = b(|t|^{-1}),\quad \\
&\tilde \gamma(t) = \gamma(|t|^{-1}),\quad 
\tilde \beta(t) = \beta(|t|^{-1}),\quad
\tilde \e(t) = \e(|t|^{-1}),\quad \tilde\Gamma_{k}(t,y) = \tilde\beta(t) (\ek\cdot y) - \frac{\tilde b(t)}4 |y|^2,
\end{align*}
so that from~\eqref{b:tuO},
\be\label{b:tuOV}
\begin{aligned} 
&\left|\tilde z(t)- \frac{2}{\kappa}|\log|t||\right|\lesssim \log|\log|t||,\quad
 \left| \tilde \lambda(t) - |t||\log|t||^{-1}\right| \lesssim |\log|t||^{-\frac 32} ,\\
& \left|\tilde b(t) -   |t| |\log|t||^{-3}\right| + |\tilde\beta(t)|
+\|\tilde \e(t)\|_{H^1} \lesssim |t| |\log|t||^{-\frac 72},
\quad |\tilde a(t)|\lesssim |t| |\log|t||^{-1}. \\
\end{aligned}
\ee 
We see from~\eqref{modu} that the pseudo-conformal transform $v(t)$ of $u(t)$ as defined
in~\eqref{pseduocno} satisfies
\[
v(t,x) = e^{-i\frac {|x|^2}{4|t|}}w(t,x),\quad
w(t,x)=\frac{e^{i\tilde\gamma(t)}}{\tilde\lambda(t)}\left(
\sum_k \big[e^{i\tilde\Gamma_{k}} Q_{\tilde a} \big]\left( \frac{x }{\tilde\lambda(t)}-  \tilde z(t)\ek\right) 
+ \tilde\e\left(t, \frac{x}{\tilde\lambda(t)} \right)\right).
\]
Note in particular that $\tilde \lambda(t) \tilde z(t) \sim \frac{2}{\kappa} |t|$ as $t\uparrow 0$.
From this, it follows   that 
\[
|v(t,x)|^2 \rightharpoonup K\|Q\|_{L^2}^2 \delta_0 \quad \hbox{as $t\uparrow 0$.}
\]
Finally, since $\nabla v(t,x) = e^{-i\frac {|x|^2}{4|t|}}\left(\nabla w -  i \frac x{2|t|} w \right)(t,x)$,
and as $t\uparrow 0$,
\begin{align*}
& \frac {1}{|t|^2}\int  {|x|^2}  |w(t,x)|^2 dx 
\lesssim \left|\frac{\tilde \lambda(t)}{t}\right|^2
 \int  \Big|\sum_k \big[e^{i\tilde\Gamma_{k}} Q_{\tilde a} \big] \left(y-  \tilde z(t)\ek\right) 
+ \tilde\e\left(t, y \right)  \Big|^2 |y|^2 dy \lesssim 1,
\\ & 
\int |\nabla w(t,x)|^2 dx\sim K \|\nabla Q\|_{L^2}^2 |t|^{-2} |\log|t||^2,
\end{align*}
we obtain~\eqref{logmodifiy}. Note that $\int |x|^2 |v(t,x)|^2 \lesssim t^2$ implies by \eqref{globalvirial} that 
$\int |x|^2 |v(t,x)|^2 = t^2 E(v)$. Thus, $E(v)>0$.

\end{document}